\documentclass[12pt]{article}
\usepackage[english]{babel}
\usepackage{amsmath,amssymb,graphicx,enumerate,latexsym,theorem}
\usepackage[a4paper,top=2cm,bottom=2cm,left=1.5cm,right=1.5cm,marginparwidth=1.75cm]{geometry}
\newcommand{\Beta}{\mathrm{B}}
\newcommand{\Eta}{\mathrm{H}}
\newcommand{\assign}{:=}
\newcommand{\downsquigarrow}{{\mbox{\rotatebox[origin=c]{-90}{$\rightsquigarrow$}}}}
\newcommand{\mathd}{\mathrm{d}}
\newcommand{\mathpi}{\pi}
\newcommand{\nin}{\not\in}
\newcommand{\textdots}{...}
\newcommand{\tmdummy}{$\mbox{}$}
\newcommand{\tmmathbf}[1]{\ensuremath{\boldsymbol{#1}}}
\newcommand{\tmop}[1]{\ensuremath{\operatorname{#1}}}
\newcommand{\tmscript}[1]{\text{\scriptsize{$#1$}}}
\newcommand{\tmstrong}[1]{\textbf{#1}}
\newcommand{\tmtextit}[1]{\text{{\itshape{#1}}}}
\newenvironment{enumeratealphacap}{\begin{enumerate}[A.] }{\end{enumerate}}
\newenvironment{enumerateromancap}{\begin{enumerate}[I.] }{\end{enumerate}}
\newenvironment{itemizedot}{\begin{itemize} }{\end{itemize}}
\newenvironment{itemizeminus}{\begin{itemize} }{\end{itemize}}
\newenvironment{proof}{\noindent\textbf{Proof\ }}{\hspace*{\fill}$\Box$\medskip}
\newcounter{nnacknowledgments}

{\theorembodyfont{\rmfamily}\newtheorem{acknowledgments*}[nnacknowledgments]{Acknowledgments}}
\newcounter{nnconvention}

{\theorembodyfont{\rmfamily}\newtheorem{convention*}[nnconvention]{Convention}}
\newtheorem{definition}{Definition}[subsection]
{\theorembodyfont{\rmfamily}\newtheorem{example}{Example}[subsection]}
\newtheorem{lemma}{Lemma}[subsection]
\newtheorem{proposition}{Proposition}[subsection]
{\theorembodyfont{\rmfamily}\newtheorem{remark}{Remark}[subsection]}
\newcounter{nnremark}

{\theorembodyfont{\rmfamily}\newtheorem{remark*}[nnremark]{Remark}}
\newtheorem{theorem}{Theorem}[subsection]
\numberwithin{equation}{section}
\begin{document}

\title{Homotopy transfer for QFT on non-compact manifold with boundary: a case
study}

\author{Minghao Wang and Gongwang Yan}

\maketitle

\begin{abstract}
  In this work we report a homological perturbation calculation to construct
  effective theories of topological quantum mechanics on
  $\mathbb{R}_{\geqslant 0}$. Such calculation can be regarded as a
  generalization of Feynman graph computation. The resulting effective
  theories fit into derived BV algebra structure, which generalizes BV
  quantization. Besides, our construction may serve as the simplest example of
  a process called ``boundary transfer'', which may help study bulk-boundary
  correspondence.
\end{abstract}

{\tableofcontents}

\section{Introduction}

Based on Batalin-Vilkovisky (BV) formalism {\cite{BATALIN198127}}, Kevin
Costello has developed a framework {\cite{costellorenormalization}} for
perturbative quantum field theory (QFT), written in the language of
homological algebra. Once a QFT is defined in this way, there will be a set of
cochain complexes of observables encoding the gauge symmetry and interaction
of this theory. Elements in this set are labeled by renormalization scales $t
\in \mathbb{R}_{> 0}$ (if we use heat kernel or fake heat kernel to perform
renormalization), while observables at different scales are connected by
homotopic renormalization. We may call such data a ``renormalized QFT''.
\begin{eqnarray}
  & \begin{array}{|l|}
    \hline
    \text{{\hspace{0.6em}}renormalized QFT{\hspace{0.6em}}}\\
    \hline
  \end{array} \xrightarrow{\text{{\normalsize{homotopy transfer}}}}
  \begin{array}{|l|}
    \hline
    \hspace{0.95em} \text{``effective theory''} \hspace{0.95em}\\
    \hline
  \end{array} &  \nonumber\\
  & \hspace{6.5em} \downsquigarrow \text{{\scriptsize{spacetime geometry}}} &
  \label{scmtcgrphofpqftrsch}\\
  & \begin{array}{|c|}
    \hline
    \begin{array}{c}
      \text{``renormalized}\\
      \text{factorization algebra''}
    \end{array}\\
    \hline
  \end{array} \xrightarrow{\hspace{8em}} \begin{array}{|c|}
    \hline
    \begin{array}{c}
      \text{``effective}\\
      \text{factorization algebra''}
    \end{array}\\
    \hline
  \end{array} &  \nonumber
\end{eqnarray}
We are interested in two aspects of a given renormalized QFT, schematically
depicted above.

The horizontal direction of (\ref{scmtcgrphofpqftrsch}) is to ask what are the
``physical observables''. The answer is an ``effective observable complex''
quasi-isomorphic to the ``renormalized observable complexes'', while does not
depend on the renormalization scale. (The scale should be regarded as an
auxiliary parameter to make the theory well-defined.) The information of
interaction on this effective observable complex is obtained by homotopy
transfer from the renormalized observable complexes. We may say this effective
complex defines an effective theory.

\begin{remark}
  Our terminology differs from that in {\cite{costellorenormalization}}. The
  ``renormalized'' here corresponds to ``effective'' there. Roughly speaking,
  we use ``effective'' to refer to structures ``on the cohomology'' (or ``on
  the physical objects''). See also Remark \ref{20220311lbl2ndrmk} for a
  concrete comparison.
\end{remark}

As for the vertical direction of (\ref{scmtcgrphofpqftrsch}), the question is
what structure on the observables reflects the fact that QFT is local on the
spacetime manifold? In {\cite{costello_gwilliam_2016,costello_gwilliam_2021}},
Kevin Costello and Owen Gwilliam developed a formalism to construct a
factorization algebra of observable complexes from a renormalized QFT. It is
this ``renormalized factorization algebra'' that encodes spacetime locality.
Naturally, we hope to remove the scales from its data by homotopy transfer,
and construct an ``effective factorization algebra''. If we achieve this goal
for various QFT's, the effective factorization algebras are expected to
exhibit structures such as canonical quanization, vertex operator algebras,
bulk-boundary correspondence and functorial formulation of QFT.

The process of homotopy transfer mentioned above is a calculation using
homological perturbation theory. To obtain the effective theory of a QFT on a
closed manifold, this calculation implies Feynman graph formulae, and the
resulting effective theory will fit into BV formalism. (See
{\cite{doubek2018quantum}} for a formal explanation of this point using finite
dimensional toy model.) However, in order to obtain the effective
factorization algebra, we also need to figure out effective theories of QFT's
living on non-compact manifolds. In this case, the argument leading to Feynman
graph formulae fails, and the resulting effective theories can exceed the
scope of BV formalism. We will find that they fit into a structure called
derived BV algebra \ {\cite{kravchenko1999deformations}}.

Moreover, to disclose bulk-boundary correspondences using effective
factorization algebras, we have to figure out effective theories of QFT's on
manifolds with boundary. But with the presence of boundary, renormalization
has not been systematically developed yet in general. If the spacetime is
$\mathbb{H}^n$ equipped with the Euclidean metric, discussion of heat kernel
renormalization can be found in {\cite{albert2016heatmfdwtbdr}}. Later, Eugene
Rabinovich formulated the renormalized theories and factorization algebras for
field theories which are ``topological normal to the boundary''
{\cite{rabinovich2021factorization}}. Particularly, topological quantum
mechanics (TQM) on $\mathbb{R}_{\geqslant 0}$ can be constructed in the
current sense.

\subsection{Main results}

In this work, we focus on the calculation for effective theories of TQM on
$\mathbb{R}_{\geqslant 0}$. Since $\mathbb{R}_{\geqslant 0}$ is a non-compact
manifold with boundary, this simple model carries double difficulties
described in the last two paragraphs. Eventually, we construct an effective
observable complex in Theorem \ref{mainthmtqmefthysdr}. The differential of
this complex and the projection from the renormalized observable complex to
this effective complex have concrete formulae given in
(\ref{mnrsttrsfddffml46}), (\ref{mainrstpjtnfml45}), respectively. They are
simplified from those initial formulae in homological perturbation lemma, just
like Feynman graph formulae arise from homological perturbation calculation
(reviewed in Section \ref{eftintqftclsdmftsec24}). As expected, the effective
observable complex is independent of renormalization scales in the
renormalized theory.

Using Theorem \ref{mainthmtqmefthysdr}, we give three examples of such
effective theories.

Example \ref{egeftccltn44} is essentially a free theory, and reproduces known
result in {\cite[Theorem 3.4.3]{rabinovich2021factorization}}. This is the
simplest example showing how degenerate field theories arise from field
theories on manifold with boundary (we refer to {\cite{butson2016degenerate}}
for more discussion on this point). Example \ref{egeftccltn45} deals with BF
theory, and refines a conclusion in {\cite[Theorem
5.0.2]{rabinovich2021factorization}}.

As for Example \ref{egeftccltn46}, we specifically design it to explicitly
show that BV structure is not enough to describe such effective theory. A
suitable candidate structure is the derived BV algebra structure
{\cite{kravchenko1999deformations}}.

We emphasize two aspects of our calculation. The former corresponds to the
spacetime being non-compact, and the latter corresponds to the presence of
spacetime boundary.

\subsubsection*{Derived BV algebras}

Example \ref{egeftccltn46} motivates us to discuss derived BV algebras. We
base our discussion on {\cite{bandiera2020cumulants}}. For the case we
consider, Proposition \ref{hrgdrvdbvmpsm512} concludes that the renormalized
observable complexes and the specific homotopies (renormalization) between
them are objects and morphisms in the category of derived BV algebras,
respectively. Moreover, Proposition \ref{reallythelastpp519ref} concludes that
both the renormalized and effective observable complexes, together with the
quasi-isomorphisms between them also lie in the category of derived BV
algebras.

\subsubsection*{Boundary transfer}

Suppose there is a renormalized QFT on a manifold $X$ with boundary $\partial
X$, we now sketch a process to obtain a factorization algebra on $\partial X$.

The renormalized QFT should give rise to a renormalized factorization algebra
$\tmop{Obs}_T$ living on (small) tubular neighborhood $T \simeq [0,
\varepsilon) \times \partial X$ of $\partial X$. Then, we can pushforward
$\tmop{Obs}_T$ to $\partial X$ via the projection $[0, \varepsilon) \times
\partial X \rightarrow \partial X$, and find its effective version by homotopy
transfer. We call this process a ``{\tmstrong{boundary transfer}}'', and hope
that it can help study various bulk-boundary correspondences.

Guided by this consideration, we make a ``quasi-calculation'' for BF theory on
$\mathbb{R}_{\geqslant 0} \times \mathbb{R}$ associated to a unimodular Lie
algebra $\mathfrak{g}$. Instead of rigorously defining the renormalized
theory, we treat it as a TQM and apply Theorem \ref{mainthmtqmefthysdr}
ignoring singularity in the data. For the so-called B boundary condition and A
boundary condition, we obtain two ``effective observable complexes''
(\ref{fakecalcuresult2dbfcdtnb411}) and (\ref{fakecalcuresult2dbfcdtna412}).
Then, they are identified with the Chevalley-Eilenberg complex $C_{\bullet}
(\Omega^{\bullet}_c (\mathbb{R}) \otimes \mathfrak{g})$ and the
Chevalley-Eilenberg algebra $\tmop{CE} (\Omega^{\bullet} (\mathbb{R}) \otimes
\mathfrak{g})$, respectively (after taking $\hbar = 1$). This result echos the
observation (see a review {\cite{Paquettekzldlt2021cij}} and references
therein) that Koszul duality between certain algebras can appear in QFT's with
proper setting.

\subsection{Future directions}

There are many questions to study in the future, we list several of them here.

First, we set up Definition \ref{tqmwobdrcdtn35} for the renormalized theory,
but have not provided a direct criterion to determine whether a given
$I^{\partial}$ satisfies the definition. For TQM on $S^1$, such a direct
criterion has been worked out. (See {\cite[Theorem 3.10 or Theorem
3.22]{2017qtztnalgindexsiliqinli}}. A proof of the algebraic index theorem can
be obtained by studying the effective theory there, see also
{\cite{Guisixukaiidxthm2019ldd}}.) The criterion for our case should be a
modification of that one, taking into account that $\rho (I^{\partial})$ may
have boundary anomaly. Besides, Definition \ref{tqmwobdrcdtn35} excludes
boundary terms (i.e. functionals supported on the boundary) in the interaction
functional, but they are relevant to bulk-boundary correspondence and should
be considered in principle.

Second, the argument leading to Theorem \ref{mainthmtqmefthysdr} relies on the
facts that the spacetime is $1$-dimensional, and the theory is topological. We
hope to find ways to simplify homological perturbation calculations for
general configurations. It could be formulated using Costello's framework. In
the $2$D case, another possible approach is to use a geometric renormalization
method of regularized integral introduced in
{\cite{rglzditglrmsrfcmdlfmlszj2021}} (see also {\cite{guisili2021elliptic}}).

Third, we recognize that the renormalized and effective observable complexes
and their quasi-isomorphisms lie in the category of derived BV algebras, but
we have not touched consequences of this fact yet. Just like BV algebras,
derived BV algebras have quantum master equations associated to them. The
solutions of these equations behave well under derived BV algebra morphisms.
We leave these considerations for later study.

\subsection{Organization of the paper}

The paper is organized as follows.

In Section \ref{sectn2algprp220307} we briefly introduce homological
perturbation theory and BV formalism. As a warm-up, we use them to construct
effective theories of free QFT's (on arbitrary manifolds) and interactive
QFT's on closed manifolds. We perform the free theory calculation without
using Hodge decomposition, hence the result has a more general setting than
that in {\cite{doubek2018quantum}}, and can be used in Section
\ref{sectn4thycmptn220307}. In the interactive case we review the way Feynman
graph formulae arise from homotpy transfer. Moreover, these two warm-up
examples exhibit all essential homological perturbation calculations in later
sections, including the one in Proposition \ref{hopethelastlablngpp513} which
simplifies {\cite[Proposition 1.27]{bandiera2020cumulants}}.

In Section \ref{sctn3defthy220307} we review relevant constructions in
{\cite{rabinovich2021factorization}}, and set up our renormalized theories.
Their effective theories are obtained in Section \ref{sectn4thycmptn220307}.

In Section \ref{sectn5eftdrvdbvalg220307} we review the definition and
homotopy transfer of derived BV algebras based on
{\cite{bandiera2020cumulants}}, then apply them to our constructions in
previous sections.

\begin{acknowledgments*}
  We would like to thank Si Li, Eugene Rabinovich, Ruggero Bandiera, Philsang
  Yoo, Nicolai Reshetikhin, Keyou Zeng, Chi Zhang and Huixing Zhong for
  illuminating discussion. We especially thank Si Li for invaluable
  conversation and guidance on this work. This work was supported by National
  Key Research and Development Program of China (NO. 2020YFA0713000).
\end{acknowledgments*}

\begin{convention*}
  {\tmdummy}
  
  \begin{itemizedot}
    \item Let $V$ be a $\mathbb{Z}$-graded $k$-vector space. We use $V_m$ to
    denote its degree $m$ component. Given homogeneous element $a \in V_m$, we
    let $| a | = m$ be its degree.
    \begin{itemizeminus}
      \item $V [n]$ denotes the degree shifting of $V$ such that $V [n]_m =
      V_{n + m}$.
      
      \item $V^{\ast}$ denotes its linear dual such that $V^{\ast}_m =
      \tmop{Hom}_k (V_{- m}, k)$. Our base field $k$ will mainly be
      $\mathbb{R}$.
      
      \item $\tmop{Sym}^m (V)$ and $\wedge^m (V)$ denote the $m$-th power
      graded symmetric product and graded skew-symmetric product respectively.
      We also denote
      \[ \tmop{Sym} (V) \assign \bigoplus_{m \geqslant 0} \tmop{Sym}^m (V),
         \quad \widehat{\tmop{Sym}} (V) \assign \prod_{m \geqslant 0}
         \tmop{Sym}^m (V) . \]
      The latter is a graded symmetric algebra with the former being its
      subalgebra. We will omit the multiplication mark for this product in
      expressions (unless confusion occurs).
      
      \item $V [[\hbar]], V ((\hbar))$ denote formal power series and Laurent
      series respectively in a variable $\hbar$ valued in $V$.
    \end{itemizeminus}
    \item We use the Einstein summation convention throughout this work.
    
    \item We use $(\pm)_{\tmop{Kos}}$ to represent the sign factors decided by
    Koszul sign rule. We always assume this rule in dealing with graded
    objects.
    \begin{itemizeminus}
      \item Example: let $j$ be a homogeneous linear map on $V$, then
      $j^{\ast}$ denotes the induced linear map on $V^{\ast}$: for $\forall f
      \in V^{\ast}, a \in V$ being homogeneous,
      \[ j^{\ast} f (a) \assign (\pm)_{\tmop{Kos}} f (j (a)) \quad \text{with
         } (\pm)_{\tmop{Kos}} = (- 1)^{| j | | f |} \text{ here.} \]
      \item Example: let $f, g, h \in V^{\ast}$ be homogeneous elements, then
      $f \otimes g \otimes h \in (V^{\ast})^{\otimes 3}$ is regarded as an
      element in $(V^{\otimes 3})^{\ast}$: for $\forall a, b, c \in V$ being
      homogeneous,
      \[ (f \otimes g \otimes h) (a \otimes b \otimes c) \assign
         (\pm)_{\tmop{Kos}} f (a) g (b) h (c) \hspace{1.5em} \text{with }
         (\pm)_{\tmop{Kos}} = (- 1)^{| h | | a | + | h | | b | + | g | | a |}
         \text{ here.} \]
      \item Example: let $(\mathcal{A}, \cdot)$ be a graded algebra, then $[-,
      -]$ means the graded commutator, i.e, for homogeneous elements $a, b$,
      \[ [a, b] \assign a \cdot b - (\pm)_{\tmop{Kos}} b \cdot a \quad
         \text{with } (\pm)_{\tmop{Kos}} = (- 1)^{| a | | b |} \text{ here.}
      \]
    \end{itemizeminus}
    \item We fix an embedding of vector spaces $\tmop{Sym}^m (V)
    \hookrightarrow V^{\otimes m}$ by
    \[ a_1 a_2 \cdots a_m \mapsto \sum_{\sigma \in \mathbf{S}_m}
       (\pm)_{\tmop{Kos}} a_{\sigma (1)} \otimes a_{\sigma (2)} \otimes \cdots
       \otimes a_{\sigma (m)}, \]
    where $\mathbf{S}_m$ denotes the symmetric group. Accordingly, any $f_1
    f_2 \cdots f_m \in \tmop{Sym}^m (V^{\ast})$ is regarded as an element in
    $(\tmop{Sym}^m (V))^{\ast}$: for $\forall a_1 a_2 \cdots a_m \in
    \tmop{Sym}^m (V)$,
    \[ f_1 f_2 \cdots f_m (a_1 a_2 \cdots a_m) = m! \sum_{\sigma \in
       \mathbf{S}_m} (\pm)_{\tmop{Kos}} f_1 (a_{\sigma (1)}) f_2 (a_{\sigma
       (2)}) \cdots f_m (a_{\sigma (m)}) . \]
    \item We call $(V, d)$ a cochain complex if $d$ is a degree $1$ map on the
    graded vector space $V$ such that $d^2 = 0$. Such $d$ is called a
    differential. A cochain map $f : (V, d) \mapsto (W, b)$ is a degree $0$
    map from $V$ to $W$ such that $b f = f d$.
    
    \item Given a manifold $X$, we denote the space of real smooth forms by
    \[ \Omega^{\bullet} (X) = \bigoplus_k \Omega^k (X) \]
    where $\Omega^k (X)$ is the subspace of $k$-forms, lying at degree $k$.
    
    \item Now that we have mentioned differential forms, definitely we will
    work with infinite dimensional functional spaces that carry natural
    topologies. The above notions for $V$ will be generalized as follows. We
    refer the reader to {\cite{treves2006topological}} or {\cite[Appendix
    2]{costellorenormalization}} for further details. Besides, {\cite[Appendix
    A]{rabinovich2021factorization}} contains specialized discussion for
    sections of vector bundles with boundary conditions.
    \begin{itemizeminus}
      \item All topological vector spaces we consider will be nuclear and we
      still use $\otimes$ to denote the completed projective tensor product.
      For example, given two manifolds $X, Y$, we have a canonical isomorphism
      \[ C^{\infty} (X) \otimes C^{\infty} (Y) = C^{\infty} (X \times Y) . \]
      \item In the involved categories, dual space is defined to be the
      continuous linear dual, equipped with the topology of uniform
      convergence of bounded subsets. We still use $(-)^{\ast}$ to denote
      taking such duals.
    \end{itemizeminus}
  \end{itemizedot}
\end{convention*}

\section{Algebraic Preliminaries}\label{sectn2algprp220307}

In this section, we briefly introduce homological perturbation theory and BV
formalism. As a warm-up, we use them to schematically construct perturbative
effective theories of free QFT's (on arbitrary manifolds) and interactive
QFT's on closed manifolds.

\subsection{Homological perturbation theory}

A special deformation retract (SDR for short) is the following data:
\begin{equation}
  (N, b) \begin{array}{c}
    i\\
    \rightleftharpoons\\
    p
  \end{array} (M, d), K \label{sdrinidata}
\end{equation}
where $(N, b)$ and $(M, d)$ are cochain complexes, and $i, p$ are cochain maps
between them. $K$ is a degree $- 1$ map on $M$, such that
\[ p i = 1, \hspace{1.5em} i p = 1 + d K + K d, \hspace{1.5em} p K = 0,
   \hspace{1.5em} K i = 0, \hspace{1.5em} K^2 = 0. \]
Consider a perturbation $\delta$ to the differential on $M$:
\[ d_1 \assign d + \delta, \qquad d_1^2 = 0 \]
we say $\delta$ is a small perturbation if $(1 - \delta K)$ is invertible. For
example, if $\sum_{n = 0}^{+ \infty} (\delta K)^n$ is well defined on $M$,
then $\delta$ is small. Note that $(1 - \delta K)$ being invertible implies
$(1 - K \delta)$ is also invertible, because we can verify that
\[ (1 - K \delta)^{- 1} = 1 + K (1 - \delta K)^{- 1} \delta . \]
\begin{lemma}
  ``Homological Perturbation'' \label{hplhplhpl}
  
  Given an SDR as (\ref{sdrinidata}) and a small perturbation $\delta$, denote
  $A \assign (1 - \delta K)^{- 1} \delta$, the following data is also an SDR:
  \begin{equation}
    (N, b_1) \begin{array}{c}
      i_1\\
      \rightleftharpoons\\
      p_1
    \end{array} (M, d_1 = d + \delta), K_1 \label{ptbdsdrdata}
  \end{equation}
  where
  \[ b_1 = b + p A i, \hspace{1.5em} i_1 = i + K A i, \hspace{1.5em} p_1 = p +
     p A K, \hspace{1.5em} K_1 = K + K A K. \]
  We can also write
  \[ i_1 = (1 - K \delta)^{- 1} i, \hspace{1.5em} p_1 = p (1 - \delta K)^{-
     1}, \hspace{1.5em} K_1 = (1 - K \delta)^{- 1} K = K (1 - \delta K)^{- 1},
  \]
  \[ b_1 = b + p_1 \delta i = b + p \delta i_1 = p_1 d_1 i = p d_1 i_1 . \]
\end{lemma}

We refer the reader to {\cite{crainic2004perturbation}} and references therein
for proof and further discussion of this lemma.

\begin{proposition}
  In the settings of Lemma \ref{hplhplhpl}, the following three statements are
  equivalent:\label{corcorhplpresvproj}
  \begin{enumeratealphacap}
    \item $p \delta K = 0$\label{corhplpresvproj001A}
    
    \item $p_1 = p$\label{corhplpresvproj001B}
    
    \item $p \delta = p \delta i p$\label{corhplpresvproj001C}
  \end{enumeratealphacap}
\end{proposition}

\begin{proposition}
  In the settings of Lemma \ref{hplhplhpl}, the following three statements are
  equivalent:\label{corcorhplpresvinj}
  \begin{enumeratealphacap}
    \item $K \delta i = 0$\label{corhplpresvinj001A}
    
    \item $i_1 = i$\label{corhplpresvinj001B}
    
    \item $i p \delta i = \delta i$\label{corhplpresvinj001C}
  \end{enumeratealphacap}
\end{proposition}

These two propositions can be checked using the formulae in Lemma
\ref{hplhplhpl}. We leave it as an exercise.

\begin{lemma}
  ``Associativity of Homological Perturbation'' \label{1stcmpstnhpllemma}
  
  Given initial data written as (\ref{sdrinidata}), suppose there are two
  small perturbations to $d$:
  \[ d_1 \assign d + \delta_1, \qquad d_2 \assign d + (\delta_1 + \delta_2),
  \]
  where the former induces an SDR written as (\ref{ptbdsdrdata}). Then,
  $\delta_2$ must be small with respect to (\ref{ptbdsdrdata}), because we can
  verify
  \[ (1 - \delta_2 K_1)^{- 1} = (1 - \delta_1 K) (1 - (\delta_1 + \delta_2)
     K)^{- 1} . \]
  Denote
  \[ A_2 \assign (1 - (\delta_1 + \delta_2) K)^{- 1} (\delta_1 + \delta_2),
     \qquad A_2' \assign (1 - \delta_2 K_1)^{- 1} \delta_2 . \]
  For the perturbed differential $d_2$, we have the ``one-step perturbation''
  to (\ref{sdrinidata}):
  \[ (N, b_2) \begin{array}{c}
       i_2\\
       \rightleftharpoons\\
       p_2
     \end{array} (M, d_2), K_2 \]
  where
  \[ b_2 = b + p A_2 i, \hspace{1.5em} i_2 = i + K A_2 i, \hspace{1.5em} p_2 =
     p + p A_2 K, \hspace{1.5em} K_2 = K + K A_2 K. \]
  We also have the ``two-step perturbation'' to (\ref{sdrinidata}):
  \[ (N, b_2') \begin{array}{c}
       i_2'\\
       \rightleftharpoons\\
       p_2'
     \end{array} (M, d_2), K_2' \]
  where
  \[ b_2' = b_1 + p_1 A_2' i_1, \hspace{1.5em} i_2' = i_1 + K_1 A_2' i_1,
     \hspace{1.5em} p_2' = p_1 + p_1 A_2' K_1, \hspace{1.5em} K_2' = K_1 + K_1
     A_2' K_1 . \]
  Then the conclusion is that these two perturbed SDR's actually coincide.
\end{lemma}

The proof of this lemma is not hard and we leave it as an exercise.

\begin{lemma}
  \label{dualcnstrctnsdrlm25}``Dual Construction''
  
  Given data written as (\ref{sdrinidata}), take the duals, then the following
  is an SDR:
  \[ (N^{\ast}, b^{\ast}) \begin{array}{c}
       p^{\ast}\\
       \rightleftharpoons\\
       i^{\ast}
     \end{array} (M^{\ast}, d^{\ast}), \quad - K^{\ast} \]
\end{lemma}

The verification is straightforward and we skip it. This taking dual operation
commutes with homological perturbation.

\begin{lemma}
  ``Symmetric Tensor Power Construction'' (See e.g., {\cite[Section
  5]{2014hplalgoperad}}.)\label{lemsymsdrcnstrctn}
  
  Given data written as (\ref{sdrinidata}), the following is an SDR:
  \begin{equation}
    (\tmop{Sym} (N), b^{\tmop{der}}) \begin{array}{c}
      i^{\tmop{sym}}\\
      \rightleftharpoons\\
      p^{\tmop{sym}}
    \end{array} (\tmop{Sym} (M), d^{\tmop{der}}), K^{\tmop{sym}}
    \label{tnsrpwrsdr}
  \end{equation}
  where $i^{\tmop{sym}}, p^{\tmop{sym}}$ are algebraic maps extended from $i,
  p$:
  \[ i^{\tmop{sym}} = \sum_{n \geqslant 0} i^{\otimes n}, \qquad
     p^{\tmop{sym}} = \sum_{n \geqslant 0} p^{\otimes n}, \]
  and $b^{\tmop{der}}, d^{\tmop{der}}$ are derivations extended from $b, d$
  using Leibniz rule:
  \[ b^{\tmop{der}} = \sum_{n \geqslant 1} \sum_{m = 0}^{n - 1} 1^{\otimes m}
     \otimes b \otimes 1^{\otimes (n - m - 1)}, \qquad d^{\tmop{der}} =
     \sum_{n \geqslant 1} \sum_{m = 0}^{n - 1} 1^{\otimes m} \otimes d \otimes
     1^{\otimes (n - m - 1)}, \]
  and
  \begin{eqnarray*}
    K^{\tmop{sym}} & = & \sum_{n \geqslant 1} \frac{1}{n!} \sum_{\sigma \in
    \mathbf{S}_n} \sigma^{- 1} \left( \sum_{m = 0}^{n - 1} 1^{\otimes m} K
    \otimes \pi^{\otimes (n - m - 1)} \right) \sigma\\
    & = & q K^{\tmop{der}} = K^{\tmop{der}} q,
  \end{eqnarray*}
  with $\pi \assign i p$, $\sigma \in \mathbf{S}_n$ permuting the tensor
  factors of $\tmop{Sym}^n (M)$, $K^{\tmop{der}}$ being the derivation
  extended from $K$, and
  \[ q \assign \sum_{n \geqslant 1} \sum_{\epsilon \in \{ 0, 1 \}^n}^{|
     \epsilon | < n} \frac{| \epsilon | ! (n - 1 - | \epsilon |) !}{n!}
     \pi^{\epsilon_1} \otimes \pi^{\epsilon_2} \otimes \cdots \otimes
     \pi^{\epsilon_n}, \text{ with } | \epsilon | = \epsilon_1 + \cdots +
     \epsilon_n . \]

  The above statement remains unchanged if we replace $\tmop{Sym} (N)$ and
  $\tmop{Sym} (M)$ by $\widehat{\tmop{Sym}} (N)$ and $\widehat{\tmop{Sym}}
  (M)$, respectively.
\end{lemma}

It is direct to see that taking symmetric tensor power commutes with taking
dual. For the sake of brevity, in the rest of this paper {\tmstrong{we will
just use}} $b, d, i, p$ {\tmstrong{to denote}} $b^{\tmop{der}},
d^{\tmop{der}}, i^{\tmop{sym}}, p^{\tmop{sym}}$ in symmetric tensor power
constructions if there is no ambiguity.

\subsubsection*{Perturbation by conjugation}

Given an SDR as (\ref{sdrinidata}), consider a conjugation on $M$
\[ (M, d) \begin{array}{c}
     U\\
     \rightleftharpoons\\
     U^{- 1}
   \end{array} (M, U d U^{- 1}) \]
Denote $d_U \assign U d U^{- 1}$. If $(d_U - d)$ is a small perturbation with
respect to the initial data (\ref{sdrinidata}), we then write the perturbed
SDR as
\[ (N, b_U') \begin{array}{c}
     i_U'\\
     \rightleftharpoons\\
     p_U'
   \end{array} (M, d_U), K_U' \]
\begin{proposition}
  In the above setting, the following statements are
  equivalent:\label{conjuptbtnprojcomcase}
  \begin{enumeratealphacap}
    \item There is a conjugation on $N$\label{conjupropprojA}
    \[ (N, b) \begin{array}{c}
         W\\
         \rightleftharpoons\\
         W^{- 1}
       \end{array} (N, W b W^{- 1}) \]
    such that
    \[ b'_U = W b W^{- 1} \qquad \text{and} \qquad p_U' = W p U^{- 1} \]
    \item The invertible map $U$ satisfies\label{conjupropprojB}
    \[ p U^{- 1} K = 0, \qquad \text{and} \qquad (p U^{- 1} i)^{- 1} \text{
       exists on } N. \]
  \end{enumeratealphacap}
  If \ref{conjupropprojA} (hence also \ref{conjupropprojB}) holds, then there
  must be
  \[ W = (p U^{- 1} i)^{- 1} . \]
\end{proposition}

\begin{proof}
  Assume statement \ref{conjupropprojA} holds. By the formula for perturbed
  projection map in Lemma \ref{hplhplhpl}, it is easy to observe that $p'_U K
  = 0$ and $p'_U i = 1$. So $p U^{- 1} K = W^{- 1} p'_U K = 0$, $W (p U^{- 1}
  i) = p'_U i = 1$. This means statement \ref{conjupropprojB} holds, and $W =
  (p U^{- 1} i)^{- 1}$.
  
  Assume statement \ref{conjupropprojB} holds. Then $p U^{- 1} i'_U = p U^{-
  1} i$. So,
  \begin{eqnarray*}
    p'_U - (p U^{- 1} i)^{- 1} p U^{- 1} & = & (p U^{- 1} i)^{- 1} (p U^{- 1}
    i) p'_U - (p U^{- 1} i)^{- 1} p U^{- 1}\\
    & = & (p U^{- 1} i)^{- 1} p U^{- 1} i'_U p'_U - (p U^{- 1} i)^{- 1} p
    U^{- 1}\\
    & = & (p U^{- 1} i)^{- 1} p U^{- 1} (i'_U p'_U - 1)\\
    & = & (p U^{- 1} i)^{- 1} p U^{- 1} (d_U K'_U + K'_U d_U)\\
    & = & (p U^{- 1} i)^{- 1} p U^{- 1} d_U K'_U\\
    & = & (p U^{- 1} i)^{- 1} p d U^{- 1} K'_U\\
    & = & (p U^{- 1} i)^{- 1} b p U^{- 1} K'_U\\
    & = & 0.
  \end{eqnarray*}
  So, $b'_U = p'_U d_U i = (p U^{- 1} i)^{- 1} p U^{- 1} U d U^{- 1} i = (p
  U^{- 1} i)^{- 1} b (p U^{- 1} i)$. This means statement \ref{conjupropprojA}
  holds.
\end{proof}

\begin{proposition}
  In the same setting for Proposition \ref{conjuptbtnprojcomcase}, the
  following statements are equivalent:\label{conjuptbtninjcomcase}
  \begin{enumeratealphacap}
    \item There is a conjugation on $N$\label{conjupropinjA}
    \[ (N, b) \begin{array}{c}
         W\\
         \rightleftharpoons\\
         W^{- 1}
       \end{array} (N, W b W^{- 1}) \]
    such that
    \[ b'_U = W b W^{- 1} \qquad \text{and} \qquad i'_U = U i W^{- 1} \]
    \item The invertible map $U$ satisfies\label{conjupropinjB}
    \[ K U i = 0, \qquad \text{and} \qquad (p U i)^{- 1} \text{ exists on } N.
    \]
  \end{enumeratealphacap}
  If \ref{conjupropinjA} (hence also \ref{conjupropinjB}) holds, then there
  must be
  \[ W = p U i. \]
\end{proposition}

The proof is left as an exercise.

\subsection{Batalin-Vilkovisky algebras and the quantum master equation}

\begin{definition}
  \label{dfntndgbvalgbrasec2229}A {\tmstrong{differential Batalin-Vilkovisky
  (BV) algebra}} is a triple $(\mathcal{A}, Q, \Delta)$ where
  \begin{itemizedot}
    \item $\mathcal{A}$ is a $\mathbb{Z}$-graded commutative associative
    unital algebra. Assume the base field is $\mathbb{R}$.
    
    \item $Q : \mathcal{A} \mapsto \mathcal{A}$ is a derivation of degree $1$
    such that $Q^2 = 0$.
    
    \item $\Delta : \mathcal{A} \mapsto \mathcal{A}$ is a linear operator of
    degree $1$ such that $\Delta^2 = 0$, and $[Q, \Delta] = Q \Delta + \Delta
    Q = 0$.
    
    \item $\Delta$ is a ``second-order'' operator w.r.t. the product of
    $\mathcal{A}$. Precisely, define the binary operator $\{ -, - \} :
    \mathcal{A} \otimes \mathcal{A} \mapsto \mathcal{A}$ as:
    \[ \{ a, b \} \assign \Delta (a b) - (\Delta a) b - (- 1)^{| a |} a \Delta
       b, \qquad \text{for } \forall a, b \in \mathcal{A}. \]
    Then for $\forall a \in \mathcal{A}$, $\{ a, - \}$ is a derivation of
    degree $(| a | + 1)$: for $\forall b, c \in \mathcal{A}$
    \[ \{ a, b c \} = \{ a, b \} c + (\pm)_{\tmop{Kos}} b \{ a, c \}, \qquad
       \text{with } (\pm)_{\tmop{Kos}} = (- 1)^{| b | | a | + | b |} \text{
       here} . \]
  \end{itemizedot}
\end{definition}

We call $\Delta$ the BV operator. $\{ -, - \}$ is called the BV bracket, which
measures the failure of $\Delta$ being a derivation. The above description
implies the following properties:
\begin{itemizedot}
  \item $\{ a, b \} = (- 1)^{| a | | b |}  \{ b, a \}$.
  
  \item $\Delta \{ a, b \} = - \{ \Delta a, b \} - (- 1)^{| a |} \{ a, \Delta
  b \}$.
  
  \item $Q \{ a, b \} = - \{ Q a, b \} - (- 1)^{| a |} \{ a, Q b \}$.
  
  \item $\{ a, \{ b, c \} \} = (- 1)^{| a | + 1} \{ \{ a, b \}, c \} + (-
  1)^{(| a | + 1) (| b | + 1)} \{ b, \{ a, c \} \}$.
\end{itemizedot}
Let $\hbar$ be a formal variable of degree $0$ (representing the quantum
parameter), we can extend the above $Q, \Delta$ to $\mathbb{R}
[[\hbar]]$-linear operators on $\mathcal{A} [[\hbar]]$. Then, $(\mathcal{A},
Q, \Delta)$ being a differential BV algebra implies $Q + \hbar \Delta$ is a
differential on $\mathcal{A} [[\hbar]]$. There is a systematic way to twist
(i.e., perturb) this differential, sketched in the following.

\begin{definition}
  Let $(\mathcal{A}, Q, \Delta)$ be a differential BV algebra. A degree $0$
  element $I \in \mathcal{A} [[\hbar]]$ is said to satisfy quantum master
  equation (QME) if
  \begin{equation}
    Q I + \hbar \Delta I + \frac{1}{2} \{ I, I \} = 0. \label{qmesmplestcase}
  \end{equation}
\end{definition}

It is direct to check that (\ref{qmesmplestcase}) implies $(Q + \hbar \Delta +
\{ I, - \})^2 = 0$ on $\mathcal{A} [[\hbar]]$.

The ``second-order'' property of $\Delta$ allows us to write down a formally
equivalent equation of QME:
\[ (Q + \hbar \Delta) e^{I / \hbar} = 0, \]
and it implies this formal conjugation relation of operators on $\mathcal{A}
[[\hbar]]$:
\[ Q + \hbar \Delta + \{ I, - \} = e^{- I / \hbar} (Q + \hbar \Delta) e^{I /
   \hbar} . \]
Although $e^{I / \hbar}$ is not defined on $\mathcal{A} [[\hbar]]$, we can
still make sense of the above two formulae without worrying about the powers
of $\hbar^{- 1}$.

For later convenience, we give an example of differential BV algebra here. In
Lemma \ref{lemsymsdrcnstrctn} we have seen how to construct a differential
graded commutative algebra from a cochain complex
\begin{equation}
  (M, d) \rightsquigarrow (\widehat{\tmop{Sym}} (M), d) .
  \label{cochaintodgcaconstructnrep}
\end{equation}
\begin{definition}
  We call a linear operator $G$ on $\widehat{\tmop{Sym}} (M)$ a
  ``{\tmstrong{$2$-to-$0$ operator}}'', if
  \begin{itemizedot}
    \item $G (\tmop{Sym}^{\leqslant 1} (M)) = 0$, and $G (\tmop{Sym}^2 (M))
    \subset \tmop{Sym}^0 (M)$.
    
    \item For $n \geqslant 2$, $\forall m_1 m_2 \cdots m_n \in \tmop{Sym}^n
    (M)$,
    \[ G (m_1 m_2 \cdots m_n) = \sum_{i < j} (\pm)_{\tmop{Kos}} G (m_i m_j)
       m_1 \ldots \widehat{m_i} \ldots \widehat{m_j} \ldots m_n . \]
  \end{itemizedot}
\end{definition}

\begin{remark*}
  In short, a $2$-to-$0$ operator can be regarded as a second-order
  differential operator with constant coefficients.
\end{remark*}

It is easy to verify the following propositions:

\begin{proposition}
  \label{2t0oprtallcmt}Any two $2$-to-$0$ operators $G_1, G_2$ commute: $[G_1,
  G_2] = 0$. This means that any $2$-to-$0$ operator of degree $1$ is a
  differential.
\end{proposition}

\begin{proposition}
  \label{2t0oprtcmtderto2t0}For $(\widehat{\tmop{Sym}} (M), d)$ in
  (\ref{cochaintodgcaconstructnrep}) and a $2$-to-$0$ operator $G$, $[G, d]$
  is also a $2$-to-$0$ operator, decided by
  \[ [G, d] (m_1 m_2) = G d (m_1 m_2) \qquad \text{for } \forall m_1, m_2 \in
     M. \]
\end{proposition}

\begin{proposition}
  \label{freeqtmdgbveg}For $(\widehat{\tmop{Sym}} (M), d)$ in
  (\ref{cochaintodgcaconstructnrep}), let $\Delta$ be a $2$-to-$0$ operator of
  degree $1$ such that
  \[ \Delta d (m_1 m_2) = \Delta ((d m_1) m_2 + (\pm)_{\tmop{Kos}} m_1 (d
     m_2)) = 0 \qquad \text{for } \forall m_1, m_2 \in M, \]
  then $(\widehat{\tmop{Sym}} (M), d, \Delta)$ is a differential BV algebra.
\end{proposition}

We refer to the review {\cite{Libvbv2017exk}} for geometric descriptions of BV
formalism and its relations to quantum field theory.

\subsection{Application: effective theory of a free
QFT}\label{applctneftoffreeqftsecsec}

In Costello's framework, for each QFT we construct a family of differential BV
algebras which are all equivalent in the sense of homotopic renormalization.
Since renormalization specified for our main case will be discussed in later
sections, we ignore this complexity here (i.e., we only refer to structures at
a fixed scale $t$ in the renormalized theory).

So, attached to each free QFT, there is a differential BV algebra
$(\widehat{\tmop{Sym}} (M), d, \Delta)$ as described in Proposition
\ref{freeqtmdgbveg}. The meaning of $M$ is ``the set of classical linear
observables'', and usually $(M, d)$ is the dual of the field space for the
theory. Then, there is a cochain complex of quantum observables
\begin{equation}
  (\widehat{\tmop{Sym}} (M) [[\hbar]], d + \hbar \Delta) .
  \label{freeqtmobsvbdata001}
\end{equation}
Roughly speaking, an effective theory is a ``smaller'' and possibly ``more
invariant'' cochain complex quasi-isomorphic to (\ref{freeqtmobsvbdata001}).
In our scope, it is constructed using homological perturbation theory,
explained in the following.

We start from an SDR written as (\ref{sdrinidata})
\[ (N, b) \begin{array}{c}
     i\\
     \rightleftharpoons\\
     p
   \end{array} (M, d), K \]
with $N$ regarded as ``the effective classical linear observables''. Then, by
symmetric tensor power construction (Lemma \ref{lemsymsdrcnstrctn}), there is
an SDR
\begin{equation}
  (\widehat{\tmop{Sym}} (N) [[\hbar]], b) \begin{array}{c}
    i\\
    \rightleftharpoons\\
    p
  \end{array} (\widehat{\tmop{Sym}} (M) [[\hbar]], d), K^{\tmop{sym}}
  \label{clsclobssdrini001}
\end{equation}
where we have extended the maps linearly over $\mathbb{R} [[\hbar]]$. It is
direct to see that $\sum_{n = 0}^{+ \infty} (\hbar \Delta K^{\tmop{sym}})^n$
is well defined on $\widehat{\tmop{Sym}} (M) [[\hbar]]$, so $\hbar \Delta$ is
a small perturbation to this SDR.

\begin{proposition}
  \label{freethyeffcorigsdrrslt28}In above settings, we can write down the
  perturbation result of (\ref{clsclobssdrini001}) by $\hbar \Delta$ as
  follows:
  \begin{equation}
    (\widehat{\tmop{Sym}} (N) [[\hbar]], b_{\hbar}) \begin{array}{c}
      i_{\hbar}\\
      \rightleftharpoons\\
      p_{\hbar}
    \end{array} (\widehat{\tmop{Sym}} (M) [[\hbar]], d + \hbar \Delta),
    K_{\hbar} \label{freeqtmeffthyreslt}
  \end{equation}
  where
  \begin{eqnarray*}
    i_{\hbar} & = & i\\
    b_{\hbar} & = & b + \hbar p \Delta i\\
    p_{\hbar} & = & p e^{\hbar (\Delta K^{\tmop{sym}})_2}
  \end{eqnarray*}
  with $(\Delta K^{\tmop{sym}})_2$ being the $2$-to-$0$ operator decided by
  \[ (\Delta K^{\tmop{sym}})_2 (m_1 m_2) = \Delta K^{\tmop{sym}} (m_1 m_2) =
     \frac{1}{2} \Delta (K \otimes (1 + i p) + (1 + i p) \otimes K) (m_1 m_2)
     \text{\quad for } \forall m_1, m_2 \in M. \]
\end{proposition}

\begin{proof}
  By associativity of perturbation theory (Lemma \ref{1stcmpstnhpllemma}), we
  can calculate the resulting SDR by ``two-step perturbation''. First, define
  another $2$-to-$0$ operator $\Delta^{\infty}$:
  \begin{eqnarray*}
    \Delta^{\infty} (m_1 m_2) & : = & \Delta (m_1 m_2) + [\Delta
    K^{\tmop{sym}}, d] (m_1 m_2)\\
    & = & \Delta (m_1 m_2) + \Delta K^{\tmop{sym}} d (m_1 m_2)\\
    & = & \Delta (1 + K^{\tmop{sym}} d) (m_1 m_2)\\
    & = & \Delta (i p - d K^{\tmop{sym}}) (m_1 m_2)\\
    & = & \Delta i p (m_1 m_2) .
  \end{eqnarray*}
  It is direct to verify that $(\widehat{\tmop{Sym}} (M), d, \Delta^{\infty})$
  is also a differential BV algebra described by Proposition
  \ref{freeqtmdgbveg}. By construction we have
  \[ \Delta^{\infty} i = i (p \Delta i), \qquad p \Delta^{\infty} = (p \Delta
     i) p \]
  So if we use $\hbar \Delta^{\infty}$ to perturb (\ref{clsclobssdrini001}),
  the statement \ref{corhplpresvproj001A} in Proposition
  \ref{corcorhplpresvproj} and statement \ref{corhplpresvinj001A} in
  Proposition \ref{corcorhplpresvinj} are both valid, hence the perturbed
  result is:
  \begin{equation}
    (\widehat{\tmop{Sym}} (N) [[\hbar]], b + \hbar p \Delta i)
    \begin{array}{c}
      i\\
      \rightleftharpoons\\
      p
    \end{array} (\widehat{\tmop{Sym}} (M) [[\hbar]], d + \hbar
    \Delta^{\infty}), K^{\infty} \label{intermdatcalfreeqtmptbtn}
  \end{equation}
  where we do not touch the concrete formula for $K^{\infty}$, only keep in
  mind that $K^{\infty}$ can be written as $K^{\tmop{sym}} T$ or $T'
  K^{\tmop{sym}}$ with some invertible operator $T, T'$.
  
  Now, consider a conjugation on the RHS of (\ref{intermdatcalfreeqtmptbtn})
  defined by $U = e^{- \hbar (\Delta K^{\tmop{sym}})_2}$. By Proposition
  \ref{2t0oprtallcmt} and Proposition \ref{2t0oprtcmtderto2t0}, it is direct
  to see
  \begin{eqnarray*}
    &  & U (d + \hbar \Delta^{\infty}) U^{- 1}\\
    & = & (d + \hbar \Delta^{\infty}) - \hbar [(\Delta K^{\tmop{sym}})_2, (d
    + \hbar \Delta^{\infty})] + \frac{1}{2!} \hbar^2 [(\Delta
    K^{\tmop{sym}})_2, [(\Delta K^{\tmop{sym}})_2, (d + \hbar
    \Delta^{\infty})]] - \ldots\\
    & = & d + \hbar \Delta^{\infty} - \hbar [(\Delta K^{\tmop{sym}})_2, d]\\
    & = & d + \hbar \Delta .
  \end{eqnarray*}
  This conjugation indeed defines a small perturbation to
  (\ref{intermdatcalfreeqtmptbtn}), and after this second perturbation the
  differential on $\widehat{\tmop{Sym}} (M) [[\hbar]]$ becomes exactly the one
  on the RHS of (\ref{freeqtmeffthyreslt}).
  
  Since $K i = 0$, we have
  \[ U i = U^{- 1} i = i. \]
  By similar reason for Proposition \ref{2t0oprtcmtderto2t0} we can find that
  $[(\Delta K^{\tmop{sym}})_2, K^{\tmop{der}}] = 0$, so $p U^{- 1}
  K^{\tmop{der}} = 0$. Hence
  \[ p U^{- 1} K^{\infty} = 0, \quad K^{\infty} U i = 0, \quad p U i = p U^{-
     1} i = 1, \]
  which means that the statement \ref{conjupropprojB} in Proposition
  \ref{conjuptbtnprojcomcase} and statement \ref{conjupropinjB} in Proposition
  \ref{conjuptbtninjcomcase} are both valid. So after this perturbation by
  conjugation, (\ref{intermdatcalfreeqtmptbtn}) will exactly give rise to
  (\ref{freeqtmeffthyreslt}).
\end{proof}

In this way, we obtain an effective theory which also corresponds to a
differential BV algebra $(\widehat{\tmop{Sym}} (N), b, p \Delta i)$. The
injection and projection maps between the effective quantum observable complex
and the renormalized quantum observable complex are figured out. We should pay
attention to the operator $(\Delta K^{\tmop{sym}})_2$, it corresponds to
``contracting with propagator'' in physicists' language.

\subsection{Application: effective theory of an interactive QFT on closed
manifold}\label{eftintqftclsdmftsec24}

In Costello's framework, an interactive QFT which is ``in the neighborhood''
of a free QFT $(\widehat{\tmop{Sym}} (M), d, \Delta)$ will give rise to a
quantum observable complex
\begin{equation}
  (\widehat{\tmop{Sym}} (M) [[\hbar]], d + \hbar \Delta +
  \delta^{\tmop{int}}), \label{abstctgnrlintobscmplx001}
\end{equation}
with $\delta^{\tmop{int}}$ being a $\mathbb{R} [[\hbar]]$-linear derivation on
$\widehat{\tmop{Sym}} (M) [[\hbar]]$. Particularly, if the theory is
constructed on a closed manifold, then the observable complex is expected to
have the form
\begin{equation}
  (\widehat{\tmop{Sym}} (M) [[\hbar]], d + \hbar \Delta + \{ I, - \})
  \label{abstctclsdmfdintobscmplx001}
\end{equation}
where $I \in \widehat{\tmop{Sym}} (M) [[\hbar]]$ is a degree $0$ element
satisfying the QME for $(\widehat{\tmop{Sym}} (M), d, \Delta)$. We call $I$
the ``action functional'' encoding the interaction. More precisely, we require
\begin{equation}
  I \in \widehat{\tmop{Sym}} (M) [[\hbar]]^+, \label{interactprprsmall}
\end{equation}
where
\begin{equation}
  \widehat{\tmop{Sym}} (M) [[\hbar]]^+ \assign \left( \prod_{n \geqslant 3}
  \tmop{Sym}^n (M) + \hbar \widehat{\tmop{Sym}} (M) [[\hbar]] \right) \subset
  \widehat{\tmop{Sym}} (M) [[\hbar]] . \label{itactnfctnlspcwell213}
\end{equation}
This implies that, given the SDR (\ref{clsclobssdrini001}), $\sum_{n = 0}^{+
\infty} ((\hbar \Delta + \{ I, - \}) K^{\tmop{sym}})^n$ is well defined on
$\widehat{\tmop{Sym}} (M) [[\hbar]]$. (This point can be checked by analyzing
the $\hbar$-grading and symmetric tensor power grading, which is left as an
exercise.) So $\hbar \Delta + \{ I, - \}$ is a small perturbation to
(\ref{clsclobssdrini001}). To express the perturbation result we need to
introduce another notation.

\begin{definition}
  \label{HRGoprt1stand2nd}Let $G$ be a $2$-to-$0$ operator of degree $0$ on
  $\widehat{\tmop{Sym}} (M)$. We define {\tmstrong{the first homotopic
  renormalization group (HRG) operator}}
  \[ \mathcal{W} (G, -) : \widehat{\tmop{Sym}} (M) [[\hbar]]^+ \mapsto
     \widehat{\tmop{Sym}} (M) [[\hbar]]^+ \]
  by this formal formula
  \[ \mathcal{W} (G, I) \assign \hbar \log (e^{\hbar G} e^{I / \hbar}) \quad
     \text{for } I \in \widehat{\tmop{Sym}} (M) [[\hbar]]^+ . \]
  The real content of this formula is a summation over connected Feynman graph
  expansion. We refer to {\cite[Chapter 2]{costellorenormalization}} for
  details of this definition. (Note that the operator $G$ here corresponds to
  ``contracting with propagator'' in the reference, so their notation differs
  a little from ours.)
  
  Similarly, we define {\tmstrong{the second HRG operator}}
  \[ \mathcal{W} (G, -, -) : \widehat{\tmop{Sym}} (M) [[\hbar]]^+ \times
     \widehat{\tmop{Sym}} (M) [[\hbar]] \mapsto \widehat{\tmop{Sym}} (M)
     [[\hbar]] \]
  by this formal formula
  \[ \mathcal{W} (G, I, f) \assign e^{-\mathcal{W} (G, I) / \hbar} e^{\hbar G}
     (e^{I / \hbar} f) \quad \text{for } I \in \widehat{\tmop{Sym}} (M)
     [[\hbar]]^+, f \in \widehat{\tmop{Sym}} (M) [[\hbar]] . \]
  Actually this is also a summation over connected Feynman graphs, with the
  restriction that each graph should contain exactly one vertex representing
  $f$.
\end{definition}

We impose another condition (only within the current subsection) between the
BV operator $\Delta$ and the map $K$ in (\ref{sdrinidata}):
\begin{equation}
  \Delta ((K m_1) m_2) = (- 1)^{| m_1 |} \Delta (m_1 K m_2) \quad \text{for }
  \forall m_1, m_2 \in M. \label{bvoprtcptblwithhmtpyinv00}
\end{equation}
This implies
\begin{equation}
  (\Delta K^{\tmop{sym}})_2 (m_1 K m_2) = 0 \quad \text{for } \forall m_1, m_2
  \in M, \label{ppgtvanishkfactorcdt}
\end{equation}
and
\begin{equation}
  (\Delta K^{\tmop{sym}})_2 (m_1 i m_2) = 0 \quad \text{for } \forall m_1, m_2
  \in M. \label{ppgtvanishifactorcdt}
\end{equation}
\begin{remark}
  If the SDR (\ref{sdrinidata}) for classical linear observables comes from a
  Hodge decomposition of field space on closed manifold,
  (\ref{bvoprtcptblwithhmtpyinv00}) will be automatically satisfied. It is
  convenient to illustrate this point schematically by considering finite
  dimensional field space. We refer to {\cite[Section 3.2]{doubek2018quantum}}
  for Hodge decomposition of dg (degree $- 1$) symplectic vector space and how
  it induces an SDR.
\end{remark}

\begin{proposition}
  \label{intactvthyclsdmfdeffcorigsdrrslt217}In above settings, we can write
  down the perturbation result of (\ref{clsclobssdrini001}) by $\hbar \Delta +
  \{ I, - \}$ as follows:
  \begin{equation}
    (\widehat{\tmop{Sym}} (N) [[\hbar]], b_{\hbar}^{\tmop{int}})
    \begin{array}{c}
      i_{\hbar}^{\tmop{int}}\\
      \rightleftharpoons\\
      p_{\hbar}^{\tmop{int}}
    \end{array} (\widehat{\tmop{Sym}} (M) [[\hbar]], d + \hbar \Delta + \{ I,
    - \}), K_{\hbar}^{\tmop{int}} \label{endendend217real}
  \end{equation}
  where
  \begin{eqnarray*}
    b_{\hbar}^{\tmop{int}} & = & b + \hbar p \Delta i + \{ I_{\tmop{eff}}, -
    \}_{p \Delta i}\\
    p_{\hbar}^{\tmop{int}} & = & p\mathcal{W} ((\Delta K^{\tmop{sym}})_2, I,
    -),
  \end{eqnarray*}
  with $\{ -, - \}_{p \Delta i}$ being the BV bracket on $\widehat{\tmop{Sym}}
  (N)$ induced by $p \Delta i$, and
  \[ I_{\tmop{eff}} \assign p\mathcal{W} ((\Delta K^{\tmop{sym}})_2, I) \in
     \widehat{\tmop{Sym}} (N) [[\hbar]]^+ . \]
\end{proposition}

\begin{proof}
  By associativity of perturbation theory (Lemma \ref{1stcmpstnhpllemma}), we
  can use $\{ I, - \}$ to perturb (\ref{freeqtmeffthyreslt}):
  \[ (\widehat{\tmop{Sym}} (N) [[\hbar]], b + \hbar p \Delta i)
     \begin{array}{c}
       i\\
       \rightleftharpoons\\
       p_{\hbar} = p e^{\hbar (\Delta K^{\tmop{sym}})_2}
     \end{array} (\widehat{\tmop{Sym}} (M) [[\hbar]], d + \hbar \Delta),
     K_{\hbar} \]
  to find out the desired result. (Lemma \ref{1stcmpstnhpllemma} contains the
  fact that this perturbation is also small.)
  
  Recall we have a formal conjugation
  \[ d + \hbar \Delta + \{ I, - \} = e^{- I / \hbar} (d + \hbar \Delta) e^{I /
     \hbar} \]
  Although $U \assign e^{- I / \hbar}$ is not a well-defined operator on
  $\widehat{\tmop{Sym}} (M) [[\hbar]]$, we can still go through the following
  formal calculation:
  \begin{eqnarray*}
    p_{\hbar} U^{- 1} K^{\tmop{sym}} & = & p e^{\hbar (\Delta
    K^{\tmop{sym}})_2} e^{I / \hbar} K^{\tmop{sym}} = 0.
  \end{eqnarray*}
  This is a consequence of (\ref{ppgtvanishkfactorcdt}). Also by
  (\ref{ppgtvanishifactorcdt}) we have
  \begin{eqnarray*}
    p_{\hbar} U^{- 1} i & = & p e^{\hbar (\Delta K^{\tmop{sym}})_2} e^{I /
    \hbar} i\\
    & = & (p e^{\hbar (\Delta K^{\tmop{sym}})_2} e^{I / \hbar})\\
    & = & (p e^{\mathcal{W} ((\Delta K^{\tmop{sym}})_2, I) / \hbar})\\
    & = & e^{I_{\tmop{eff}} / \hbar} .
  \end{eqnarray*}
  So, formally we have verified statement \ref{conjupropprojB} in Proposition
  \ref{conjuptbtnprojcomcase}
  \[ p_{\hbar} U^{- 1} K_{\hbar} = 0, \qquad (p_{\hbar} U^{- 1} i)^{- 1} =
     e^{- I_{\tmop{eff}} / \hbar}, \]
  and we can formally write down the perturbed projection predicted by
  Proposition \ref{conjuptbtnprojcomcase}
  \begin{eqnarray*}
    (p_{\hbar} U^{- 1} i)^{- 1} p_{\hbar} U^{- 1} & = & e^{- I_{\tmop{eff}} /
    \hbar} p e^{\hbar (\Delta K^{\tmop{sym}})_2} e^{I / \hbar}\\
    & = & (p e^{-\mathcal{W} ((\Delta K^{\tmop{sym}})_2, I) / \hbar}) p
    e^{\hbar (\Delta K^{\tmop{sym}})_2} e^{I / \hbar}\\
    & = & p e^{-\mathcal{W} ((\Delta K^{\tmop{sym}})_2, I) / \hbar} e^{\hbar
    (\Delta K^{\tmop{sym}})_2} e^{I / \hbar}\\
    & = & p\mathcal{W} ((\Delta K^{\tmop{sym}})_2, I, -)
  \end{eqnarray*}
  which is exactly $p_{\hbar}^{\tmop{int}}$.
  
  Since (formally) we have
  \[ (b + \hbar p \Delta i) e^{I_{\tmop{eff}} / \hbar} = (b + \hbar p \Delta
     i) p_{\hbar} e^{I / \hbar} = p_{\hbar} (d + \hbar \Delta) e^{I / \hbar} =
     0, \]
  we can formally write down the perturbed differential on
  $\widehat{\tmop{Sym}} (N) [[\hbar]]$ predicted by Proposition
  \ref{conjuptbtnprojcomcase}
  \[ e^{- I_{\tmop{eff}} / \hbar} b_{\hbar} e^{I_{\tmop{eff}} / \hbar} = b +
     \hbar p \Delta i + \{ I_{\tmop{eff}}, - \}_{p \Delta i} \]
  which is exactly $b_{\hbar}^{\tmop{int}}$.
  
  In the current context, all the above formal arguments make (the proof of)
  Proposition \ref{conjuptbtnprojcomcase} really works, because we can handle
  the intermediate calculations by involving $\prod_{n \geqslant 0}
  (\tmop{Sym}^n (M) ((\hbar)))$ and $\prod_{n \geqslant 0} (\tmop{Sym}^n (N)
  ((\hbar)))$. This completes the proof.
\end{proof}

So we obtain an effective interactive theory, encoded by an effective action
functional $I_{\tmop{eff}}$ which is a solution to QME of the effective free
theory $(\widehat{\tmop{Sym}} (N), b, p \Delta i)$. We also obtain a concrete
formula for the projection map $p_{\hbar}^{\tmop{int}}$ from the observable
complex of the renormalized theory to that of the effective theory. Feynman
graph calculations are packaged in the formulae for $I_{\tmop{eff}}$ and
$p_{\hbar}^{\tmop{int}}$, supporting our language as a substitute for
physicists' (perturbative) path integral story.

\begin{remark}
  \label{20220311lbl2ndrmk}As mentioned in the introduction, our terminology
  ``effective'' differs from that in {\cite{costellorenormalization}}. The
  effective action $I_{\tmop{eff}}$ here is the ``restriction of scale
  $\infty$ effective action to harmonic fields'' there \ (see
  {\cite[Proposition 10.7.2]{costellorenormalization}} for details).
\end{remark}

\section{TQM on $\mathbb{R}_{\geqslant 0}$: the Renormalized
Theory}\label{sctn3defthy220307}

We have seen that, for a free QFT or an interactive QFT on closed manifold,
the effective theory can be described using the same kind of algebraic
structure of the renormalized theory. For an interactive QFT on non-compact
manifold, we can still construct an observable complex that looks like
(\ref{abstctclsdmfdintobscmplx001}), but now
\begin{equation}
  I \nin \widehat{\tmop{Sym}} (M) [[\hbar]], \label{itactnnotobsvble}
\end{equation}
i.e., the action functional is not an observable (because integration might be
divergent on a non-compact manifold). Actually this is a very common case,
because a QFT on a closed manifold can be restricted to any open subset of
this manifold (see {\cite[Section 2.14]{costellorenormalization}} and
{\cite[Section 8.7]{costello_gwilliam_2021}} for details).

On a manifold $X$ with boundary $\partial X$, if there is a QFT constructed
using our current formulation, we can then restrict it to a (small) tubular
neighborhood $T \simeq [0, \varepsilon) \times \partial X$ of the boundary
$\partial X$. Then, an effective observable complex of this QFT on $T$ might
be regarded as a system on $\partial X$. Because of (\ref{itactnnotobsvble}),
if we calculate this complex by homological perturbation, the argument leading
to Proposition \ref{intactvthyclsdmfdeffcorigsdrrslt217} fails, and the result
can exceed the scope of BV formalism. Certainly we are interested in such
results.

With the presence of spacetime boundary, renormalization has not been
systematically developed yet in general. If the spacetime is $\mathbb{H}^n$
equipped with the Euclidean metric, discussion of heat kernel renormalization
can be found in {\cite{albert2016heatmfdwtbdr}}. Later, Eugene Rabinovich
formulated the renormalized theories and factorization algebras for field
theories which are ``topological normal to the boundary''
{\cite{rabinovich2021factorization}}. If we restrict such a theory on $X$ to a
tubular neighborhood of $\partial X$, the result will be equivalent to a QFT
on $\mathbb{R}_{\geqslant 0} \times \partial X$. If $\partial X$ is a point,
this is a topological quantum mechanics, which is the main case we study in
this work.

In this section, relevant constructions will be extracted from
{\cite{rabinovich2021factorization}}, making up the renormalized theories.
Effective theories will be calculated in the next section.

\subsubsection*{Content of field space}

Let $(V, Q^{\partial})$ be a cochain complex of finite dimensional vector
space. By Leibniz rule, the differential $Q^{\partial}$ induces differentials
on various tensors of $V, V^{\ast}$, still denoted by $Q^{\partial}$. Let
\begin{equation}
  \omega^{\partial} \in \wedge^2 (V^{\ast}), \quad \text{s.t. } Q^{\partial}
  \omega^{\partial} = 0 \label{dgsplctccptbcdtn}
\end{equation}
be a degree $0$ symplectic pairing on $V$ compatible with $Q^{\partial}$. Let
$L, L'$ be two Lagrangian subspaces of $V$ satisfying
\begin{equation}
  V = L \oplus L', \quad Q^{\partial} (L) \subseteq L, \quad Q^{\partial} (L')
  \subseteq L' . \label{dgsplctcplrztncptb}
\end{equation}
For $v \in V$, the map $v \mapsto \omega^{\partial} (v, -)$ induces a vector
space isomorphism $V \simeq V^{\ast}$, hence also a vector space isomorphism
$\wedge^2 V \simeq \wedge^2 (V^{\ast})$. Let
\[ K^{\partial} \in \wedge^2 (V) \]
be the image of $\omega^{\partial}$ under $\wedge^2 (V^{\ast}) \simeq \wedge^2
V$. If we regard $K^{\partial}$ as an element in
\[ V^{\otimes 2} = (L \otimes L) \oplus (L' \otimes L') \oplus (L \otimes L')
   \oplus (L' \otimes L), \]
we can write
\begin{equation}
  K^{\partial} = K^{\partial}_- + K^{\partial}_+, \quad K^{\partial}_- \in L
  \otimes L', K^{\partial}_+ \in L' \otimes L, \text{ s.t. } \sigma
  K^{\partial}_- = - K^{\partial}_+, \label{splctckrnldecmpstn}
\end{equation}
where $\sigma$ permutes the two factors of $V^{\otimes 2}$. By
(\ref{dgsplctccptbcdtn}),
\begin{equation}
  Q^{\partial} K^{\partial}_- = Q^{\partial} K^{\partial}_+ = 0.
  \label{closedcdtnofsmplctckernel3501}
\end{equation}
We fix the spacetime manifold to be $\mathbb{R}_{\geqslant 0}$. Let
$\iota^{\ast} : \Omega^{\bullet} (\mathbb{R}_{\geqslant 0}) \otimes V \mapsto
V$ denote the pullback of $V$-valued forms induced by the inclusion map $\iota
: \{ 0 \} \hookrightarrow \mathbb{R}_{\geqslant 0}$. The {\tmstrong{field
space}} of our theory is
\begin{equation}
  \mathcal{E}_L \assign \{ f \in \Omega^{\bullet} (\mathbb{R}_{\geqslant 0})
  \otimes V| \iota^{\ast} f \in L \} . \label{fildspcdfntn036}
\end{equation}
Note that we have chosen a boundary condition in above definition. The set of
classical linear observables is the dual space $\mathcal{E}_L^{\ast}$ (we have
mentioned the meaning of dual in the convention part), and
\[ \mathcal{O} (\mathcal{E}_L) \assign \widehat{\tmop{Sym}}
   (\mathcal{E}_L^{\ast}) \]
is the set of all classical observables. By construction we have a surjection
\[ \widehat{\tmop{Sym}} ((\Omega^{\bullet} (\mathbb{R}_{\geqslant 0}) \otimes
   V)^{\ast}) \mapsto \mathcal{O} (\mathcal{E}_L) . \]
There is a differential $Q \assign \mathd + Q^{\partial}$ on $\mathcal{E}_L$,
where $\mathd$ is the de Rham differential on $\Omega^{\bullet}
(\mathbb{R}_{\geqslant 0})$. $Q$ induces differentials on
$\mathcal{E}_L^{\ast}$ and $\mathcal{O} (\mathcal{E}_L)$, still denoted by
$Q$.

There is a subcomplex of $(\mathcal{E}_L, Q)$
\begin{equation}
  \left( \mathcal{E}_{L, c} \assign \left\{ f \in \mathcal{E}_L |f \text{ is
  compactly supported} \right\}, Q \right) . \label{cpctsupfldspc111}
\end{equation}
$\omega^{\partial}$ can be $\Omega^{\bullet} (\mathbb{R}_{\geqslant
0})$-linearly extended to a map $\mathcal{E}_L \times \mathcal{E}_L \mapsto
\Omega^{\bullet} (\mathbb{R}_{\geqslant 0})$. (The wedge product on
$\Omega^{\bullet} (\mathbb{R}_{\geqslant 0})$ is implicitly used.) Then, we
have a degree $- 1$ pairing:
\[ \int_{\mathbb{R}_{\geqslant 0}} \omega^{\partial} : \enspace
   \mathcal{E}_{L, c} \times \mathcal{E}_L \mapsto \mathbb{R} \qquad (f, g)
   \mapsto \int_{\mathbb{R}_{\geqslant 0}} \omega^{\partial} (f, g) . \]
It induces an embedding of cochain complex
\begin{equation}
  (\mathcal{E}_{L, c} [1], - Q) \hookrightarrow (\mathcal{E}_L^{\ast}, Q)
  \qquad \eta f \mapsto \int_{\mathbb{R}_{\geqslant 0}} \omega^{\partial} (f,
  -), \label{embedofsmrobsvbtoobsvb}
\end{equation}
where $\eta$ is a formal variable of degree $- 1$, and we use $\eta f$ to
represent the element in $\mathcal{E}_{L, c} [1]$ corresponding to $f \in
\mathcal{E}_{L, c}$. This embedding is actually a quasi-isomorphism (see
{\cite[Appendix A.3]{rabinovich2021factorization}} for details). So, we have a
set of ``smeared observables'' quasi-isomorphic to $\mathcal{O}
(\mathcal{E}_L)$:
\[ \mathcal{O}_{\tmop{sm}} (\mathcal{E}_L) \assign \widehat{\tmop{Sym}}
   (\mathcal{E}_{L, c} [1]) . \]
Let $\Delta_0$ be the $2$-to-$0$ operator on $\mathcal{O}_{\tmop{sm}}
(\mathcal{E}_L)$ decided by
\[ \Delta_0 (\eta f_1 \eta f_2) = \int_{\mathbb{R}_{\geqslant 0}}
   \omega^{\partial} ((- 1)^{| f_1 |} f_1, f_2) . \]
Then it is easy to verify:

\begin{proposition}
  \label{smrobsvbdfbvalg}$(\mathcal{O}_{\tmop{sm}} (\mathcal{E}_L), - Q,
  \Delta_0)$ is a differential BV algebra.
\end{proposition}

\subsubsection*{Free quantum observable complex}

In order to incorporate interaction, we have to define BV structure directly
on $\mathcal{O} (\mathcal{E}_L)$. For simplicity, we will not mention the
``doubling trick'' method in {\cite[Chapter 4]{rabinovich2021factorization}}.
Expressions for the BV operator and propagator here should be equivalent to
those in {\cite[Chapter 5.1]{rabinovich2021factorization}}.

We fix a metric on $\mathbb{R}_{\geqslant 0}$ by $\langle \partial_x,
\partial_x \rangle = 1$, where $x$ is the coordinate. Let $\mathd^{\tmop{GF}}$
denotes the Hodge dual to the de Rham operator induced by this metric:
\[ \text{d}^{\tmop{GF}} (f \tmop{dx}) = - \partial_x f \qquad \text{for } f
   \in \Omega^0 (\mathbb{R}_{\geqslant 0}) . \]
The heat kernels, being smooth functions on $\mathbb{R}_{> 0} \times
(\mathbb{R}_{\geqslant 0} \times \mathbb{R}_{\geqslant 0})$, \tmtextit{}are
the following:
\[ H_D (t, x, y) = \frac{1}{\sqrt{4 \pi t}} \left( e^{- \frac{(x - y)^2}{4 t}}
   - e^{- \frac{(x + y)^2}{4 t}} \right), \qquad H_N (t, x, y) =
   \frac{1}{\sqrt{4 \pi t}} \left( e^{- \frac{(x - y)^2}{4 t}} + e^{- \frac{(x
   + y)^2}{4 t}} \right), \]
where the indice ``$D$'' means Dirichlet boundary condition, and indice
``$N$'' means Neumann boundary condition. Let $H_t \assign H_D (t, x, y)
\tmop{dy} - H_N (t, x, y) \tmop{dx}$, then it is direct to verify
\begin{eqnarray}
  \left( 1 \otimes \text{d} + \text{d} \otimes 1 \right) H_t & = & 0
  \hspace{5em}  \label{hkderhamclosedcdt37}\\
  - \frac{1}{2}  \left( 1 \otimes \text{d} + \text{d} \otimes 1 \right) \left(
  \text{d}^{\tmop{GF}} \otimes 1 + 1 \otimes \mathd^{\tmop{GF}} \right) H_t &
  = & \partial_t H_t . \nonumber
\end{eqnarray}
So, for $\forall \varepsilon, \Lambda \in \mathbb{R}_{> 0}$,
\begin{equation}
  H_{\Lambda} = H_{\varepsilon} - \frac{1}{2}  \left( 1 \otimes \text{d} +
  \text{d} \otimes 1 \right) \left( \text{d}^{\tmop{GF}} \otimes 1 + 1 \otimes
  \mathd^{\tmop{GF}} \right) \int_{\varepsilon}^{\Lambda} \tmop{dt} H_t
  \label{truehtknlppty123} .
\end{equation}
Let $\phi \in C^{\infty} (\mathbb{R})$ be a compactly supported even function
which evaluates to $1$ in a neighborhood of $\{ 0 \}$. Define
\[ \widetilde{H_t} \assign \frac{1}{\sqrt{4 \pi t}} \left( \phi (x - y) e^{-
   \frac{(x - y)^2}{4 t}} (\tmop{dy} - \tmop{dx}) - \phi (x + y) e^{- \frac{(x
   + y)^2}{4 t}} (\tmop{dy} + \tmop{dx}) \right), \]
then
\[ \widetilde{H_t} - H_t = \frac{1}{\sqrt{4 \pi t}} \left( (\phi (x - y) - 1)
   e^{- \frac{(x - y)^2}{4 t}} (\tmop{dy} - \tmop{dx}) - (\phi (x + y) - 1)
   e^{- \frac{(x + y)^2}{4 t}} (\tmop{dy} + \tmop{dx}) \right) \]
can be extended to a smooth form on $\mathbb{R}_{\geqslant 0} \times
(\mathbb{R}_{\geqslant 0} \times \mathbb{R}_{\geqslant 0})$, vanishing at $t =
0$. For $t > 0$, define the {\tmstrong{BV kernel}} to be
\begin{eqnarray}
  K_t & \assign & \frac{1}{2} \left( H_t - \frac{1}{2}  \left( 1 \otimes
  \text{d} + \text{d} \otimes 1 \right) \left( \text{d}^{\tmop{GF}} \otimes 1
  + 1 \otimes \mathd^{\tmop{GF}} \right) \int_0^t \tmop{ds} (\widetilde{H_s} -
  H_s) \right) \otimes K^{\partial}_+ \nonumber\\
  &  & - \frac{1}{2} \sigma \left( H_t - \frac{1}{2}  \left( 1 \otimes
  \text{d} + \text{d} \otimes 1 \right) \left( \text{d}^{\tmop{GF}} \otimes 1
  + 1 \otimes \mathd^{\tmop{GF}} \right) \int_0^t \tmop{ds} (\widetilde{H_s} -
  H_s) \right) \otimes K^{\partial}_-  \label{bvknldfntn311}
\end{eqnarray}
where the $\sigma$ here permutes variables $x$ and $y$, $K^{\partial}_-,
K^{\partial}_+$ are defined in (\ref{splctckrnldecmpstn}). By a little bit
calculus, we can verify that $K_t \in \tmop{Sym}^2 (\mathcal{E}_L)$, if the
tensor factors are properly rearranged. Due to
(\ref{closedcdtnofsmplctckernel3501}), (\ref{hkderhamclosedcdt37}), $K_t$ is
$Q$-closed. Define the {\tmstrong{propagator}} to be
\begin{eqnarray}
  P (\varepsilon, \Lambda) & \assign & \left( \frac{- 1}{4}  \left(
  \text{d}^{\tmop{GF}} \otimes 1 + 1 \otimes \mathd^{\tmop{GF}} \right)
  \int_{\varepsilon}^{\Lambda} \tmop{dt} \widetilde{H_t} \right) \otimes
  K^{\partial}_+ \nonumber\\
  &  & + \sigma \left( \frac{1}{4}  \left( \text{d}^{\tmop{GF}} \otimes 1 + 1
  \otimes \mathd^{\tmop{GF}} \right) \int_{\varepsilon}^{\Lambda} \tmop{dt}
  \widetilde{H_t} \right) \otimes K^{\partial}_-  \label{ppgtdfntn31212}
\end{eqnarray}
where $\varepsilon, \Lambda \in \mathbb{R}_{> 0}$. Similarly, $P (\varepsilon,
\Lambda) \in \tmop{Sym}^2 (\mathcal{E}_L)$. By (\ref{truehtknlppty123}), we
have
\begin{equation}
  K_{\Lambda} = K_{\varepsilon} + \left( 1 \otimes \text{d} + \text{d} \otimes
  1 \right) P (\varepsilon, \Lambda) \label{rgofbvknlandppgt310} .
\end{equation}
For $\forall E \in \tmop{Sym}^2 (\mathcal{E}_L)$, let $\partial_E$ denote the
$2$-to-$0$ operator on $\mathcal{O} (\mathcal{E}_L)$ decided by
\[ \partial_E (a_1 a_2) = (- 1)^{| E | (| a_1 | + | a_2 |)} a_1 a_2 (E) \qquad
   \text{for } a_1, a_2 \in \mathcal{E}_L^{\ast} . \]
It is direct to verify:

\begin{proposition}
  \label{freernmlzdobsvblcmplxtpthy32}Denote $\partial_{K_t}$ by $\Delta_t$,
  then,
  \begin{itemizedot}
    \item $(\mathcal{O} (\mathcal{E}_L), Q, \Delta_t)$ is a differential BV
    algebra for $\forall t > 0$.
    
    \item The embedding $\mathcal{O}_{\tmop{sm}} (\mathcal{E}_L)
    \hookrightarrow \mathcal{O} (\mathcal{E}_L)$ induced by
    (\ref{embedofsmrobsvbtoobsvb}) makes $(\mathcal{O}_{\tmop{sm}}
    (\mathcal{E}_L), - Q, \Delta_t)$ also a differential BV algebra. Moreover,
    the $t \rightarrow 0$ limit of $\Delta_t$ exists on
    $\mathcal{O}_{\tmop{sm}} (\mathcal{E}_L)$, and equals to $\Delta_0$ in
    Proposition \ref{smrobsvbdfbvalg}.
    
    \item For $\forall \varepsilon, \Lambda > 0$, $\Delta_{\Lambda} =
    \Delta_{\varepsilon} + [\partial_{P (\varepsilon, \Lambda)}, Q]$.
    Equivalently, there is such a conjugation of cochain complexes
    \[ (\mathcal{O} (\mathcal{E}_L) [[\hbar]], Q + \hbar \Delta_{\varepsilon})
       \begin{array}{c}
         e^{\hbar \partial_{P (\varepsilon, \Lambda)}}\\
         \rightleftharpoons\\
         e^{- \hbar \partial_{P (\varepsilon, \Lambda)}}
       \end{array} (\mathcal{O} (\mathcal{E}_L) [[\hbar]], Q + \hbar
       \Delta_{\Lambda}) . \]
  \end{itemizedot}
\end{proposition}

\begin{remark}
  \label{bvknlandppgtpropersptppty1101}Regarded as forms on
  $\mathbb{R}_{\geqslant 0} \times \mathbb{R}_{\geqslant 0}$, $K_t$ and $P
  (\varepsilon, \Lambda)$ have proper supports. Namely, each of the projection
  maps $\pi_1, \pi_2 : \mathbb{R}_{\geqslant 0} \times \mathbb{R}_{\geqslant
  0} \mapsto \mathbb{R}_{\geqslant 0}$ is proper when restricted to
  $\tmop{supp} (K_t)$ or $\tmop{supp} (P (\varepsilon, \Lambda))$. For $P
  (\varepsilon, \Lambda)$ this is by construction, because $\tmop{supp}
  (\widetilde{H_t})$ is bounded by a proper neighborhood of the diagonal,
  uniformly for all $t$. For $K_t$, if we take $\varepsilon \rightarrow 0$ in
  (\ref{rgofbvknlandppgt310}):
  \[ K_t = K_0 + \left( 1 \otimes \text{d} + \text{d} \otimes 1 \right) P (0,
     t), \]
  we obtain a relation in the space of distributions. Since $K_0, P (0, t)$
  have proper supports, so does $K_t$.
\end{remark}

\subsubsection*{Interaction}

We have introduced the functional space $\mathcal{O} (\mathcal{E}_L)$, and
obtained a family of free observable complexes $\{ (\mathcal{O}
(\mathcal{E}_L) [[\hbar]], Q + \hbar \Delta_t) |t \in \mathbb{R}_{> 0} \}$. We
can also write down functionals on the space $\mathcal{E}_{L, c}$:
\[ \mathcal{O} (\mathcal{E}_{L, c}) \assign \widehat{\tmop{Sym}}
   (\mathcal{E}_{L, c}^{\ast}) . \]
Then, there is a natural inclusion $\mathcal{O} (\mathcal{E}_L)
\hookrightarrow \mathcal{O} (\mathcal{E}_{L, c})$. A functional $f \in
\tmop{Sym}^n (\mathcal{E}_{L, c}^{\ast})$ with $n > 0$ belongs to $\mathcal{O}
(\mathcal{E}_L)$ if $f$ has compact support. Here we regard the input of $f$
as forms on $\mathbb{R}_{\geqslant 0}^n$. The support of $f$ means
$\mathbb{R}_{\geqslant 0}^n \backslash\mathcal{U}$, where $\mathcal{U}$ is the
union of all open sets $U \subset \mathbb{R}_{\geqslant 0}^n$ satisfying $f
(e) = 0$ whenever $e$ has support on a compact subset of $U$. For $n > 0$,
define
\[ \mathcal{O}^n_{\mathcal{P}} (\mathcal{E}_{L, c}) \assign \left\{ f \in
   \tmop{Sym}^n (\mathcal{E}_{L, c}^{\ast}) | \tmop{supp} (f) \text{ is
   proper} \right\}, \]
where $\tmop{supp} (f)$ is proper if each of the $n$ projection maps
$\mathbb{R}_{\geqslant 0}^n \mapsto \mathbb{R}_{\geqslant 0}$ is proper when
restricted to $\tmop{supp} (f)$. Since a compact support is also proper,
$\tmop{Sym}^n (\mathcal{E}_L^{\ast}) \subset \mathcal{O}^n_{\mathcal{P}}
(\mathcal{E}_{L, c})$. $\mathcal{O} (\mathcal{E}_{L, c})$ has a subspace
\[ \mathcal{O}_{\mathcal{P}}^{> 0} (\mathcal{E}_{L, c}) \assign \prod_{n > 0}
   \mathcal{O}_{\mathcal{P}}^n (\mathcal{E}_{L, c}) \]
which is not a subring: for $f_1, f_2 \in \mathcal{O}_{\mathcal{P}}^{> 0}
(\mathcal{E}_{L, c})$, $f_1 f_2 \in \mathcal{O} (\mathcal{E}_{L, c})$ may not
have proper support.

\begin{remark}
  \label{rmk34biabiaddd}Due to Remark \ref{bvknlandppgtpropersptppty1101}, we
  have the following facts:
  \begin{itemizedot}
    \item $\Delta_t, \partial_{P (\varepsilon, \Lambda)}$ are well defined on
    $\mathcal{O}_{\mathcal{P}}^{> 0} (\mathcal{E}_{L, c})$.
    
    \item The BV bracket induced by $\Delta_t$:
    \[ \{ -, - \}_t : \mathcal{O} (\mathcal{E}_L) \otimes \mathcal{O}
       (\mathcal{E}_L) \mapsto \mathcal{O} (\mathcal{E}_L) \]
    can be extended to a map $\{ -, - \}_t : \mathcal{O}_{\mathcal{P}}^{> 0}
    (\mathcal{E}_{L, c}) \otimes \mathcal{O} (\mathcal{E}_L) \mapsto
    \mathcal{O} (\mathcal{E}_L)$. This is a derivation with respect to
    $\mathcal{O} (\mathcal{E}_L)$. We can further extend $\{ -, - \}_t$ to
    \[ \{ -, - \}_t : \mathcal{O}_{\mathcal{P}}^{> 0} (\mathcal{E}_{L, c})
       \otimes \mathcal{O}_{\mathcal{P}}^{> 0} (\mathcal{E}_{L, c}) \mapsto
       \mathcal{O}_{\mathcal{P}}^{> 0} (\mathcal{E}_{L, c}), \]
    which makes $\mathcal{O}_{\mathcal{P}}^{> 0} (\mathcal{E}_{L, c}) [- 1]$ a
    graded Lie algebra.
  \end{itemizedot}
  (The reasoning behind these facts is similar to the proof of {\cite[Chapter
  2, Lemma 14.5.1]{costellorenormalization}}.)
\end{remark}

So, although $\mathcal{O}_{\mathcal{P}}^{> 0} (\mathcal{E}_{L, c})$ is not an
graded algebra, the QME ``{\tmstrong{at scale $t$}}''
\[ Q I + \hbar \Delta_t I + \frac{1}{2} \{ I, I \}_t = 0 \]
is well defined for degree $0$ elements in $\mathcal{O}_{\mathcal{P}}^{> 0}
(\mathcal{E}_{L, c}) [[\hbar]]$. If $I$ is a solution, $Q + \hbar \Delta_t +
\{ I, - \}_t$ will be a differential on $\mathcal{O} (\mathcal{E}_L)
[[\hbar]]$, and we can formally write
\begin{equation}
  Q + \hbar \Delta_t + \{ I, - \}_t = e^{- I / \hbar} (Q + \hbar \Delta_t)
  e^{I / \hbar} . \label{fmlitactndiffopnclsedcase312}
\end{equation}

By definition (\ref{fildspcdfntn036}), we have a decomposition for the field
space
\begin{equation}
  \mathcal{E}_L = (\Omega^{\bullet} (\mathbb{R}_{\geqslant 0}) \otimes L)
  \oplus \{ f \in \Omega^{\bullet} (\mathbb{R}_{\geqslant 0}) \otimes L' |
  \iota^{\ast} f = 0 \} . \label{fldspcdcpstnwithdrham313}
\end{equation}
Let $\mathcal{L}, \mathcal{L}'$ denote the first and the second component,
respectively. Define
\[ \mathcal{O}_{\mathcal{P}} (\mathcal{E}_{L, c}) [[\hbar]]^+ \assign \left(
   \mathcal{O}_{\mathcal{P}}^2 (\mathcal{L}'_c) + \prod_{n \geqslant 3}
   \mathcal{O}_{\mathcal{P}}^n (\mathcal{E}_{L, c}) + \hbar
   \mathcal{O}_{\mathcal{P}}^{> 0} (\mathcal{E}_{L, c}) [[\hbar]] \right)
   \subset \mathcal{O}_{\mathcal{P}}^{> 0} (\mathcal{E}_{L, c}) [[\hbar]], \]
where $\mathcal{O}_{\mathcal{P}}^2 (\mathcal{L}'_c) \subset
\mathcal{O}_{\mathcal{P}}^2 (\mathcal{E}_{L, c})$ are functionals only
depending on $\mathcal{L}'$. Such a quadratic component at $\hbar^0$-order is
not included in (\ref{itactnfctnlspcwell213}) to ensure well-defined
perturbation, but the presence here will cause no problem. The
Definition-Lemma 4.4.6 in {\cite{rabinovich2021factorization}} tells us these
facts:
\begin{itemizedot}
  \item The first HRG operator is well defined:
  \[ \mathcal{W} (\partial_{P (\varepsilon, \Lambda)}, -) :
     \mathcal{O}_{\mathcal{P}} (\mathcal{E}_{L, c}) [[\hbar]]^+ \mapsto
     \mathcal{O}_{\mathcal{P}} (\mathcal{E}_{L, c}) [[\hbar]]^+, \]
  formally, for $I \in \mathcal{O}_{\mathcal{P}} (\mathcal{E}_{L, c})
  [[\hbar]]^+$,
  \begin{equation}
    \mathcal{W} (\partial_{P (\varepsilon, \Lambda)}, I) = \hbar \log
    (e^{\hbar \partial_{P (\varepsilon, \Lambda)}} e^{I / \hbar}) - \hbar
    (\log (e^{\hbar \partial_{P (\varepsilon, \Lambda)}} e^{I / \hbar})) (0),
    \label{fmldfntof1sthrgopnclsedcase312}
  \end{equation}
  where the second term is the (ill-defined) constant part of the first term.
  
  This operator has Feynman graph summation formula modified from that in
  Definition \ref{HRGoprt1stand2nd}, by discarding graphs without external
  leg.
  
  \item The second HRG operator is well defined:
  \[ \mathcal{W} (\partial_{P (\varepsilon, \Lambda)}, -, -) :
     \mathcal{O}_{\mathcal{P}} (\mathcal{E}_{L, c}) [[\hbar]]^+ \times
     \mathcal{O} (\mathcal{E}_L) [[\hbar]] \mapsto \mathcal{O} (\mathcal{E}_L)
     [[\hbar]], \]
  formally, for $I \in \mathcal{O}_{\mathcal{P}} (\mathcal{E}_{L, c})
  [[\hbar]]^+$,
  \begin{equation}
    \mathcal{W} (\partial_{P (\varepsilon, \Lambda)}, I, -) = (e^{\hbar
    \partial_{P (\varepsilon, \Lambda)}} e^{I / \hbar})^{- 1} e^{\hbar
    \partial_{P (\varepsilon, \Lambda)}} e^{I / \hbar} (-) .
    \label{fmldfntnof2ndhrgoprtopnclsdcs313}
  \end{equation}
  Its Feynman graph summation formula is the same with that in Definition
  \ref{HRGoprt1stand2nd}.
\end{itemizedot}
The HRG operators satisfy certain associativity. For $\Lambda_1, \Lambda_2,
\Lambda_3 > 0, I \in \mathcal{O}_{\mathcal{P}} (\mathcal{E}_{L, c})
[[\hbar]]^+$,
\begin{eqnarray}
  \mathcal{W} (\partial_{P (\Lambda_1, \Lambda_3)}, I) & = & \mathcal{W}
  (\partial_{P (\Lambda_2, \Lambda_3)}, \mathcal{W} (\partial_{P (\Lambda_1,
  \Lambda_2)}, I)), \\
  \mathcal{W} (\partial_{P (\Lambda_1, \Lambda_3)}, I, -) & = & \mathcal{W}
  (\partial_{P (\Lambda_2, \Lambda_3)}, \mathcal{W} (\partial_{P (\Lambda_1,
  \Lambda_2)}, I), \mathcal{W} (\partial_{P (\Lambda_1, \Lambda_2)}, I, -)) . 
  \label{2ndhrgasspptysec3319}
\end{eqnarray}
\begin{lemma}
  Given a degree $0$ functional $I_{\varepsilon} \in \mathcal{O}_{\mathcal{P}}
  (\mathcal{E}_{L, c}) [[\hbar]]^+$ satisfying the QME at a certain scale
  $\varepsilon > 0$:
  \[ Q I_{\varepsilon} + \hbar \Delta_{\varepsilon} I_{\varepsilon} +
     \frac{1}{2} \{ I_{\varepsilon}, I_{\varepsilon} \}_{\varepsilon} = 0, \]
  let $I_{\Lambda}$ denote $\mathcal{W} (\partial_{P (\varepsilon, \Lambda)},
  I_{\varepsilon})$ for another scale $\Lambda > 0$. Then, $I_{\Lambda}$
  satisfies the QME at scale $\Lambda$. In addition, we have a conjugation of
  cochain complexes
  \begin{equation}
    \normalsize{(\mathcal{O} (\mathcal{E}_L) [[\hbar]], Q + \hbar
    \Delta_{\varepsilon} + \{ I_{\varepsilon}, - \}_{\varepsilon})}
    \begin{array}{c}
      \normalsize{\mathcal{W} (\partial_{P (\varepsilon, \Lambda)},
      I_{\varepsilon}, -)}\\
      \rightleftharpoons\\
      \normalsize{\mathcal{W} (\partial_{P (\Lambda, \varepsilon)},
      I_{\Lambda}, -)}
    \end{array} \normalsize{(\mathcal{O} (\mathcal{E}_L) [[\hbar]], Q + \hbar
    \Delta_{\Lambda} + \{ I_{\Lambda}, - \}_{\Lambda})} .
    \label{cjgtnofqtmobsvblcplxopnclsdcase3121}
  \end{equation}
\end{lemma}

The above properties are parallel with those for HRG operators in
{\cite{costellorenormalization,costello_gwilliam_2021}}, and can be proved
similarly using the formal formulae (\ref{fmlitactndiffopnclsedcase312}),
(\ref{fmldfntof1sthrgopnclsedcase312}) and
(\ref{fmldfntnof2ndhrgoprtopnclsdcs313}). While using these formal
expressions, we need to involve $\prod_{n \geqslant 0} (\tmop{Sym}^n
(\mathcal{E}_{L, c}^{\ast}) ((\hbar)))$ and discard ill-defined constant
factors to make sense of intermediate calculations. We leave the details to
the interested reader.

Given a function $f \in \tmop{Sym}^n (V^{\ast})$ with $n > 0$, we can extend
it $\Omega^{\bullet} (\mathbb{R}_{\geqslant 0})$-linearly and then integrate
over $\mathbb{R}_{\geqslant 0}$, thus obtain a functional in $\tmop{Sym}^n
(\mathcal{E}_{L, c}^{\ast})$. Denote this functional by $\rho (f)$, it is
supported on the diagonal of $\mathbb{R}^n_{\geqslant 0}$. In this way we have
a degree $- 1$ map:
\begin{equation}
  \rho : \quad \tmop{Sym}^2 ((L')^{\ast}) \oplus \prod_{n \geqslant 3}
  \tmop{Sym}^n (V^{\ast}) \oplus \hbar \prod_{n > 0} \tmop{Sym}^n (V^{\ast})
  [[\hbar]] \mapsto \mathcal{O}_{\mathcal{P}} (\mathcal{E}_{L, c}) [[\hbar]]^+
  \label{akszactnfnctnlcstrctn314}
\end{equation}

Now, we are ready to define the concrete theory to study.

\begin{definition}
  \label{tqmwobdrcdtn35}A UV finite topological quantum mechanics (TQM) on
  $\mathbb{R}_{\geqslant 0}$ {\tmstrong{with only bulk interaction}} consists
  of the following:
  \begin{itemizedot}
    \item a family of differential BV algebras $\{ (\mathcal{O}
    (\mathcal{E}_L), Q, \Delta_t) |t \in \mathbb{R}_{> 0} \}$ stated in
    Proposition \ref{freernmlzdobsvblcmplxtpthy32};
    
    \item a ``scale $0$'' action functional $\rho (I^{\partial}) \in
    \mathcal{O}_{\mathcal{P}} (\mathcal{E}_{L, c}) [[\hbar]]^+$, where
    $I^{\partial} \in \widehat{\tmop{Sym}} (V^{\ast}) [[\hbar]]$ has degree
    $1$, $\rho$ is the map in (\ref{akszactnfnctnlcstrctn314}).
  \end{itemizedot}
  They are required to satisfy these conditions:
  \begin{itemizedot}
    \item (UV finiteness) For a given scale $t > 0$ (hence for all scales),
    the following limit exists:
    \[ I_t \assign \lim_{\varepsilon \rightarrow 0} \mathcal{W} (\partial_{P
       (\varepsilon, t)}, \rho (I^{\partial})) . \]
    \item This limit $I_t$ is a solution to the QME at scale $t$.
  \end{itemizedot}
\end{definition}

\begin{remark}
  For simplicity, we impose the UV finiteness by hand. Detailed analysis may
  disclose that this condition is automatically satisfied, as in the case of
  TQM on $S^1$ (see {\cite{2017qtztnalgindexsiliqinli}}). We leave this
  consideration for later work.
\end{remark}

\begin{example}
  \label{eg36thycnstctn}Given $(V = L \oplus L', Q^{\partial},
  \omega^{\partial})$ as in (\ref{dgsplctccptbcdtn}) and
  (\ref{dgsplctcplrztncptb}), let $Q^{\tmop{rel}} : L' \mapsto L$ be a degree
  $1$ map, such that
  \[ Q^{\tmop{rel}} Q^{\partial} + Q^{\partial} Q^{\tmop{rel}} = 0, \text{ and
     } \omega^{\partial} (Q^{\tmop{rel}} a, b) + (- 1)^{| a |}
     \omega^{\partial} (a, Q^{\tmop{rel}} b) = 0 \]
  for $a, b \in V$. Consider the degree $- 1$ function
  \[ I^{\partial} : V \otimes V \mapsto \mathbb{R}, \qquad I^{\partial} (a, b)
     : = - \omega^{\partial} (Q^{\tmop{rel}} a, b) \text{ for } a, b \in V, \]
  it is direct to verify that actually $I^{\partial} \in \tmop{Sym}^2
  ((L')^{\ast}) \subset (V^{\ast})^{\otimes 2}$.
  
  Recall we have a decomposition $\mathcal{E}_L =\mathcal{L} \oplus
  \mathcal{L}'$ in (\ref{fldspcdcpstnwithdrham313}). The operators
  $\partial_{P (\varepsilon, t)}, \Delta_t, \{ -, - \}_t$ have a common
  feature: they always pair a linear functional factor depending on
  $\mathcal{L}$ with a linear functional factor depending on $\mathcal{L}'$.
  The functional
  \[ \rho (I^{\partial}) = - \int_{\mathbb{R}_{\geqslant 0}} \omega^{\partial}
     (Q^{\tmop{rel}} -, -) \in \mathcal{O}_{\mathcal{P}}^2 (\mathcal{L}'_c) \]
  only depends on $\mathcal{L}'$, hence satisfies
  \[ \partial_{P (\varepsilon, t)} \rho (I^{\partial}) = 0, \quad \Delta_t
     \rho (I^{\partial}) = 0, \quad \{ \rho (I^{\partial}), \rho
     (I^{\partial}) \}_t = 0 \]
  for $\forall \varepsilon, t > 0$. It is also easy to verify
  \[ \mathd \rho (I^{\partial}) = Q^{\partial} \rho (I^{\partial}) = 0. \]
  So, $I_t \assign \lim_{\varepsilon \rightarrow 0} \mathcal{W} (\partial_{P
  (\varepsilon, t)}, \rho (I^{\partial})) = \rho (I^{\partial})$, and $I_t$
  satisfies the QME at scale $t$, thus defines a theory in our sense.
\end{example}

\begin{example}
  \label{eg37thycnstctn}``BF theory with B boundary condition''
  
  Let $\mathfrak{g}$ be a Lie algebra with basis $\{ t^a \}_{a = 1}^{\ell}$,
  $[t^a, t^b] = f^{a b}_c t^c$. We impose the unimodular condition
  \begin{equation}
    f^{c b}_c = 0. \label{unmdlcdtnlalg}
  \end{equation}
  (We have used the Einstein summation convention.) Let
  \[ L \assign (\mathfrak{g}  [1])^{\ast} = (\mathfrak{g}^{\ast}) [- 1],
     \qquad L' \assign \mathfrak{g} [1] . \]
  For $\beta \in \mathfrak{g}^{\ast}$, we use $\epsilon \beta$ to denote the
  element in $L$ corresponding to $\beta$, where $\epsilon$ is a formal
  variable of degree $1$; similarly for $\alpha \in \mathfrak{g}$ we have
  $\eta \alpha \in L'$ where $\eta$ is a formal variable of degree $- 1$.
  There is a degree $0$ symplectic pairing $\omega^{\partial}$ on $V = L
  \oplus L'$, decided by
  \[ \omega^{\partial} (\epsilon t_a, \eta t^b) \assign \delta_a^b \]
  where $\{ t_a \}_{a = 1}^{\ell}$ is the basis of $\mathfrak{g}^{\ast}$ dual
  to $\{ t^a \}_{a = 1}^{\ell}$. Then, $(V, Q^{\partial} \assign 0,
  \omega^{\partial})$ satisfies (\ref{dgsplctccptbcdtn}) and
  (\ref{dgsplctcplrztncptb}).
  
  We define $I^{\partial} \in \tmop{Sym}^3 (V^{\ast}) \subset (V^{\otimes
  3})^{\ast}$ to be the function decided by
  \[ I^{\partial} (\epsilon \beta, \eta \alpha_1, \eta \alpha_2) \assign -
     \beta ([\alpha_1, \alpha_2]), \qquad \text{for } \epsilon \beta \in L,
     \eta \alpha_1, \eta \alpha_2 \in L' . \]
  Let $\{ B^a \}_{a = 1}^{\ell}$ be the basis of $L^{\ast}$ and $\{ A_a \}_{a
  = 1}^{\ell}$ be the basis of $(L')^{\ast}$ so that
  \[ B^a (\epsilon t_b) = A_b (\eta t^a) = \delta_b^a . \]
  Then we can write
  \[ I^{\partial} = \frac{1}{2} f^{a b}_c B^c A_a A_b . \]
  So, $\rho (I^{\partial}) = \frac{1}{2} \int_{\mathbb{R} \geqslant 0} f^{a
  b}_c B^c \wedge A_a \wedge A_b$ is the interaction action functional of BF
  theory on $\mathbb{R}_{\geqslant 0}$. The boundary condition we choose is
  the so-called B condition. This data indeed fits into Definition
  \ref{tqmwobdrcdtn35}, and we refer to {\cite[Section
  5.2]{rabinovich2021factorization}} for the proof.
\end{example}

\begin{example}
  \label{eg38thycnstctn}Let $\mu_1, \mu_2, \nu_1, \nu_2$ be formal variables
  of degree $0, 1, 0, - 1$, respectively. Define
  \[ L \assign \mathbb{R} \mu_1 \oplus \mathbb{R} \mu_2, \qquad L' \assign
     \mathbb{R} \nu_1 \oplus \mathbb{R} \nu_2 . \]
  There is a degree $0$ symplectic pairing $\omega^{\partial}$ on $V \assign L
  \oplus L'$, decided by
  \[ \omega^{\partial} (\mu_i, \nu_j) = \delta_{i j}, \quad \omega^{\partial}
     (\nu_i, \nu_j) = \omega^{\partial} (\mu_i, \mu_j) = 0, \qquad i, j \in \{
     1, 2 \} . \]
  Denote the dual space of $L, L'$ by
  \[ L^{\ast} =\mathbb{R}q_1 \oplus \mathbb{R}q_2, \qquad (L')^{\ast} \assign
     \mathbb{R}p_1 \oplus \mathbb{R}p_2 \]
  with $q_i (\mu_j) = p_i (\nu_j) = \delta_{i j}$ for $i, j \in \{ 1, 2 \}$.
  Then, $(V, Q^{\partial} \assign 0, \omega^{\partial})$ satisfies
  (\ref{dgsplctccptbcdtn}) and (\ref{dgsplctcplrztncptb}). Let
  \[ I^{\partial} \assign - p_1^2 p_2, \]
  then $\rho (I^{\partial}) = - \int_{\mathbb{R}_{\geqslant 0}} p_1 \wedge p_1
  \wedge p_2$ is a degree $0$ functional depending only on $\mathcal{L}'$. For
  the same reason as in Example \ref{eg36thycnstctn},
  \[ I_t \assign \lim_{\varepsilon \rightarrow 0} \mathcal{W} (\partial_{P
     (\varepsilon, t)}, \rho (I^{\partial})) = \rho (I^{\partial}) \]
  satisfies the QME at scale $t$. So we have defined a theory.
\end{example}

\section{TQM on $\mathbb{R}_{\geqslant 0}$: the Effective
Theory}\label{sectn4thycmptn220307}

\subsection{Homological perturbation
construction}\label{hplcnstctnefttqmans41}

Given a TQM as in Definition \ref{tqmwobdrcdtn35}, we have a family of quantum
observable complexes
\begin{equation}
  \{ (\mathcal{O} (\mathcal{E}_L) [[\hbar]], Q + \hbar \Delta_t + \{ I_t, -
  \}_t) |t \in \mathbb{R}_{> 0} \} \label{tqmwobdrcrctnobscplx41}
\end{equation}
with $I_t = \lim_{\varepsilon \rightarrow 0} \mathcal{W} (\partial_{P
(\varepsilon, t)}, \rho (I^{\partial}))$ satisfying the QME at scale $t$.
These cochain complexes at different scales are related by conjugations
(\ref{cjgtnofqtmobsvblcplxopnclsdcase3121}). To construct an effective
observable complex by homological perturbation, we need to start with an SDR
whose RHS can be perturbed to elements of (\ref{tqmwobdrcrctnobscplx41}).

\begin{lemma}
  For the the field space $(\mathcal{E}_L = \{ f \in \Omega^{\bullet}
  (\mathbb{R}_{\geqslant 0}) \otimes V| \iota^{\ast} f \in L \}, Q = \mathd +
  Q^{\partial})$, we have an SDR:
  \begin{equation}
    (L, Q^{\partial}) \begin{array}{c}
      \tau\\
      \rightleftharpoons\\
      \iota^{\ast}
    \end{array} (\mathcal{E}_L, Q), \varkappa
    \label{freefieldsdropenclsdcase42}
  \end{equation}
  where $\tau, \varkappa$ both have images in $\Omega^0 (\mathbb{R}_{\geqslant
  0}) \otimes V$: for $l \in L, f \in \mathcal{E}_L, x \in
  \mathbb{R}_{\geqslant 0}$,
  \[ (\tau (l)) (x) = l, \qquad (\varkappa (f)) (x) = - \int_0^x f. \]
\end{lemma}

The check is left to the reader as an exercise.

We denote the dual construction (Lemma \ref{dualcnstrctnsdrlm25}) of
(\ref{freefieldsdropenclsdcase42}) by
\[ (L^{\ast}, Q^{\partial}) \begin{array}{c}
     i\\
     \rightleftharpoons\\
     p
   \end{array} (\mathcal{E}_L^{\ast}, Q), K. \]
It further induces an SDR by symmetric tensor power construction:
\begin{equation}
  (\mathcal{O} (L) [[\hbar]], Q^{\partial}) \begin{array}{c}
    i\\
    \rightleftharpoons\\
    p
  \end{array} (\mathcal{O} (\mathcal{E}_L) [[\hbar]], Q), K^{\tmop{sym}},
  \label{sttpntsdrforourcs43}
\end{equation}
where $\mathcal{O} (L) \assign \widehat{\tmop{Sym}} (L^{\ast})$.

The homological degree of an element in $\mathcal{O} (\mathcal{E}_L)
[[\hbar]]$ is the sum of two parts, the first comes from the grading of
$\Omega^0 (\mathbb{R}_{\geqslant 0})$, and the second comes from the grading
of $V$. In the following we will call the first part the ``{\tmstrong{de Rham
degree}}''. Similarly, we can define the de Rham degree of an operator on
$\mathcal{O} (\mathcal{E}_L) [[\hbar]]$. Denote this degree by $| \alpha
|_{\tmop{dR}}$ for $\alpha$ being a functional or an operator. It is easy to
see that
\begin{itemizedot}
  \item elements in $\mathcal{O} (\mathcal{E}_L) [[\hbar]]$ have nonpositive
  de Rham degree;
  
  \item for $f \in \mathcal{O} (L) [[\hbar]]$, $| i (f) |_{\tmop{dR}} = 0$,
  elements in $p^{- 1} (f)$ also have de Rham degree $0$;
  
  \item $| Q^{\partial} |_{\tmop{dR}} = 0, | \mathd |_{\tmop{dR}} = 1, |
  K^{\tmop{sym}} |_{\tmop{dR}} = - 1$;
  
  \item $| \Delta_t |_{\tmop{dR}} = | \{ -, - \}_t |_{\tmop{dR}} = 1$;
  
  \item $| \partial_{P (\varepsilon, \Lambda)} |_{\tmop{dR}} = 0$;
  
  \item $| \rho (f) |_{\tmop{dR}} = - 1$ for $f \in \tmop{Sym}^n (V^{\ast})$
  with $n > 0$. ($\rho$ is the map in (\ref{akszactnfnctnlcstrctn314}).)
\end{itemizedot}
Now, pick a scale $t > 0$, we use
\[ \hbar \Delta_t + \delta^{\tmop{int}}_t \assign \hbar \Delta_t + \{ I_t, -
   \}_t \]
to perturb (\ref{sttpntsdrforourcs43}), whose RHS will be changed to
$(\mathcal{O} (\mathcal{E}_L) [[\hbar]], Q + \hbar \Delta_t + \{ I_t, -
\}_t)$. By analyzing the $\hbar$-grading and symmetric tensor power grading,
we can observe that $\sum_{n = 0}^{+ \infty} ((\hbar \Delta_t +
\delta^{\tmop{int}}_t) K^{\tmop{sym}})^n$ is well defined on $\mathcal{O}
(\mathcal{E}_L) [[\hbar]]$. So, $\hbar \Delta_t + \delta^{\tmop{int}}_t$ is a
small perturbation. (This is similar to the implication following
(\ref{itactnfctnlspcwell213}), although here $I_t$ can contain terms in
$\mathcal{O}_{\mathcal{P}}^2 (\mathcal{L}'_c)$.)

\begin{theorem}
  \label{mainthmtqmefthysdr}Let $\{ (\mathcal{O} (\mathcal{E}_L), Q, \Delta_t)
  |t \in \mathbb{R}_{> 0} \}$ and $\rho (I^{\partial}) \in
  \mathcal{O}_{\mathcal{P}} (\mathcal{E}_{L, c}) [[\hbar]]^+$ encode a TQM on
  $\mathbb{R}_{\geqslant 0}$ as in Definition \ref{tqmwobdrcdtn35}. At any
  scale $t > 0$, we have an SDR manifesting an effective theory for this TQM.
  Concretely, this SDR is the result of using $\hbar \Delta_t +
  \delta^{\tmop{int}}_t$ to perturb (\ref{sttpntsdrforourcs43}), written as
  follows:
  \begin{equation}
    (\mathcal{O} (L) [[\hbar]], Q^{\partial} + b_t^{\tmop{int}})
    \begin{array}{c}
      i_t^{\tmop{int}}\\
      \rightleftharpoons\\
      p_t
    \end{array} (\mathcal{O} (\mathcal{E}_L) [[\hbar]], Q + \hbar \Delta_t +
    \delta^{\tmop{int}}_t), K_t^{\tmop{int}} \label{ptbtnrsttqmopnclsd44}
  \end{equation}
  where
  \begin{eqnarray}
    b_t^{\tmop{int}} & = & p e^{\hbar (\Delta_t K^{\tmop{sym}})_2}
    \delta^{\tmop{int}}_t i.  \label{mnrsttrsfddffml46}\\
    p_t & = & p e^{\hbar (\Delta_t K^{\tmop{sym}})_2}, 
    \label{mainrstpjtnfml45}
  \end{eqnarray}
  Moreover, for any two different scales $\varepsilon, \Lambda > 0$,
  \begin{eqnarray*}
    p_{\varepsilon} & = & p_{\Lambda} \mathcal{W} (\partial_{P (\varepsilon,
    \Lambda)}, I_{\varepsilon}, -)\\
    b_{\varepsilon}^{\tmop{int}} & = & b_{\Lambda}^{\tmop{int}} .
  \end{eqnarray*}
  Namely, we obtain an effective observable complex $(\mathcal{O} (L)
  [[\hbar]], Q^{\partial} + b^{\tmop{int}})$ independent of the scale we pick
  to construct it.
\end{theorem}

\begin{proof}
  By associativity of perturbation theory (Lemma \ref{1stcmpstnhpllemma}), we
  do the calculation in two steps. First, we use only $\hbar \Delta_t$ to
  perturb (\ref{sttpntsdrforourcs43}). By Proposition
  \ref{freethyeffcorigsdrrslt28} the result can be written as
  \begin{equation}
    (\mathcal{O} (L) [[\hbar]], Q^{\partial}) \begin{array}{c}
      i\\
      \rightleftharpoons\\
      p e^{\hbar (\Delta_t K^{\tmop{sym}})_2}
    \end{array} (\mathcal{O} (\mathcal{E}_L) [[\hbar]], Q + \hbar \Delta_t),
    K_{\hbar} . \label{freetqmopnclsdsdrrst45}
  \end{equation}
  Note that the differential on $\mathcal{O} (L) [[\hbar]]$ is not perturbed,
  because $\Delta_t i = 0$ here. Then, we use $\delta^{\tmop{int}}_t$ to
  perturb (\ref{freetqmopnclsdsdrrst45}). By definition,
  \[ \delta^{\tmop{int}}_t = \{ I_t, - \}_t = \lim_{\varepsilon \rightarrow 0}
     \{ \mathcal{W} (\partial_{P (\varepsilon, t)}, \rho (I^{\partial})), -
     \}_t . \]
  Since
  \[ | \rho (I^{\partial}) |_{\tmop{dR}} = - 1, \quad | \partial_{P
     (\varepsilon, t)} |_{\tmop{dR}} = 0, \]
  it is direct to conclude from (\ref{fmldfntof1sthrgopnclsedcase312}) that,
  terms in $\mathcal{W} (\partial_{P (\varepsilon, t)}, \rho (I^{\partial}))$
  will have de Rham degree at most $- 1$. Besides, $| \{ -, - \}_t
  |_{\tmop{dR}} = 1$, so given $f \in \mathcal{O} (\mathcal{E}_L) [[\hbar]]$
  with specific de Rham degree $| f |_{\tmop{dR}}$, terms in
  $\delta^{\tmop{int}}_t (f)$ will have de Rham degree at most $| f
  |_{\tmop{dR}}$. By
  \[ | K^{\tmop{sym}} |_{\tmop{dR}} = - 1, \quad | \Delta_t |_{\tmop{dR}} = 1,
  \]
  we have $| (\Delta_t K^{\tmop{sym}})_2 |_{\tmop{dR}} = 0$. Recall that $p
  (f) = 0$ for $f \in \mathcal{O} (\mathcal{E}_L) [[\hbar]]$ with $| f
  |_{\tmop{dR}} < 0$, so
  \[ p e^{\hbar (\Delta_t K^{\tmop{sym}})_2} \delta^{\tmop{int}}_t
     K^{\tmop{sym}} = 0. \]
  This implies $p e^{\hbar (\Delta_t K^{\tmop{sym}})_2} \delta^{\tmop{int}}_t
  K_{\hbar} = 0$. Namely, perturbation $\delta^{\tmop{int}}_t$ to
  (\ref{freetqmopnclsdsdrrst45}) validates the statement
  \ref{corhplpresvproj001A} in Proposition \ref{corcorhplpresvproj}, hence
  gives rise to exactly the formulae (\ref{mnrsttrsfddffml46}) and
  (\ref{mainrstpjtnfml45}).
  
  By the same counting as above, given $f \in \mathcal{O} (\mathcal{E}_L)
  [[\hbar]]$ with specific $| f |_{\tmop{dR}}$,
  (\ref{fmldfntnof2ndhrgoprtopnclsdcs313}) implies
  \[ \mathcal{W} (\partial_{P (\varepsilon, \Lambda)}, I_{\varepsilon}, f) =
     e^{\hbar \partial_{P (\varepsilon, \Lambda)}} f + \ldots, \]
  where terms in ``{\textdots}'' have de Rham degree less than $| f
  |_{\tmop{dR}}$. So,
  \begin{eqnarray*}
    p_{\Lambda} \mathcal{W} (\partial_{P (\varepsilon, \Lambda)},
    I_{\varepsilon}, -) & = & p e^{\hbar (\Delta_{\Lambda} K^{\tmop{sym}})_2}
    e^{\hbar \partial_{P (\varepsilon, \Lambda)}}\\
    & = & p e^{\hbar ((\Delta_{\Lambda} K^{\tmop{sym}})_2 + \partial_{P
    (\varepsilon, \Lambda)})} .
  \end{eqnarray*}
  Note that $\Delta_{\Lambda} = \Delta_{\varepsilon} + [\partial_{P
  (\varepsilon, \Lambda)}, Q] = \Delta_{\varepsilon} + [\partial_{P
  (\varepsilon, \Lambda)}, \mathd]$, we have
  \[ (\Delta_{\Lambda} K^{\tmop{sym}})_2 + \partial_{P (\varepsilon, \Lambda)}
     = (\Delta_{\varepsilon} K^{\tmop{sym}})_2 + ([\partial_{P (\varepsilon,
     \Lambda)}, \mathd] K^{\tmop{sym}})_2 + \partial_{P (\varepsilon,
     \Lambda)} . \]
  It is easy to verify that $\partial_{P (\varepsilon, \Lambda)} i = 0$, and
  for $f_1, f_2 \in \mathcal{E}_L^{\ast}$, $\partial_{P (\varepsilon,
  \Lambda)} (f_1 K f_2) = 0$ (by de Rham degree reason). So
  \begin{eqnarray*}
    &  & (([\partial_{P (\varepsilon, \Lambda)}, \mathd] K^{\tmop{sym}})_2 +
    \partial_{P (\varepsilon, \Lambda)}) (f_1 f_2)\\
    & = & (\partial_{P (\varepsilon, \Lambda)} \mathd K^{\tmop{sym}} +
    \partial_{P (\varepsilon, \Lambda)}) (f_1 f_2)\\
    & = & \partial_{P (\varepsilon, \Lambda)} (i p - K^{\tmop{sym}} \mathd)
    (f_1 f_2)\\
    & = & 0,
  \end{eqnarray*}
  which means $([\partial_{P (\varepsilon, \Lambda)}, \mathd]
  K^{\tmop{sym}})_2 + \partial_{P (\varepsilon, \Lambda)} = 0$ because this is
  a $2$-to-$0$ operator. Hence we have proved $p_{\Lambda} \mathcal{W}
  (\partial_{P (\varepsilon, \Lambda)}, I_{\varepsilon}, -) =
  p_{\varepsilon}$. As a consequence,
  \begin{eqnarray*}
    Q^{\partial} + b_{\varepsilon}^{\tmop{int}} & = & p_{\varepsilon} (Q +
    \hbar \Delta_{\varepsilon} + \delta^{\tmop{int}}_{\varepsilon})
    i_{\varepsilon}^{\tmop{int}}\\
    & = & p_{\Lambda} \mathcal{W} (\partial_{P (\varepsilon, \Lambda)},
    I_{\varepsilon}, -) (Q + \hbar \Delta_{\varepsilon} +
    \delta^{\tmop{int}}_{\varepsilon}) i_{\varepsilon}^{\tmop{int}}\\
    & = & p_{\Lambda} (Q + \hbar \Delta_{\Lambda} +
    \delta^{\tmop{int}}_{\Lambda}) \mathcal{W} (\partial_{P (\varepsilon,
    \Lambda)}, I_{\varepsilon}, -) i_{\varepsilon}^{\tmop{int}}\\
    & = & (Q^{\partial} + b_{\Lambda}^{\tmop{int}}) p_{\Lambda} \mathcal{W}
    (\partial_{P (\varepsilon, \Lambda)}, I_{\varepsilon}, -)
    i_{\varepsilon}^{\tmop{int}}\\
    & = & (Q^{\partial} + b_{\Lambda}^{\tmop{int}}) p_{\varepsilon}
    i_{\varepsilon}^{\tmop{int}}\\
    & = & Q^{\partial} + b_{\Lambda}^{\tmop{int}} .
  \end{eqnarray*}
  So we have finished.
\end{proof}

\begin{remark}
  In the proof we have seen that, if we had set $Q^{\partial} = 0$ in the
  beginning, before turning on the interaction we would obtain a complex (the
  LHS of (\ref{freetqmopnclsdsdrrst45})) with $0$ differential. Then, the
  interaction could give rise to a nonzero $b^{\tmop{int}}$ on it (see Example
  \ref{egeftccltn46}). This is different from the case in Section
  \ref{eftintqftclsdmftsec24}. There, the interaction induces a formal
  conjugation on the effective observable complex, so will preserve a zero
  differential.
\end{remark}

We can actually simplify formula (\ref{mnrsttrsfddffml46}) a little. Formally
we have
\[ \delta^{\tmop{int}}_t = \{ I_t, - \}_t = \hbar e^{- I_t / \hbar} \{ e^{I_t
   / \hbar}, - \}_t = \hbar e^{- I_t / \hbar} \lim_{\varepsilon \rightarrow 0}
   \{ e^{\hbar \partial_{P (\varepsilon, t)}} e^{\rho (I^{\partial}) / \hbar},
   - \}_t . \]
So, by de Rham degree argument,
\begin{equation}
  b^{\tmop{int}} = p e^{\hbar (\Delta_t K^{\tmop{sym}})_2}
  \delta^{\tmop{int}}_t i = \lim_{\varepsilon \rightarrow 0} p e^{\hbar
  (\Delta_t K^{\tmop{sym}})_2} \{ e^{\hbar \partial_{P (\varepsilon, t)}} \rho
  (I^{\partial}), - \}_t i. \label{splfdtrsfddiffeffthy48}
\end{equation}
For later convenience, we rewrite the BV kernel $K_t$ (defined in
(\ref{bvknldfntn311})) as
\[ K_t (x, y) = (\Eta^{10}_t (x, y) \tmop{dx} + \Eta^{01}_t (x, y) \tmop{dy})
   \otimes K^{\partial}_+ + \sigma (\Eta^{10}_t (x, y) \tmop{dx} + \Eta^{01}_t
   (x, y) \tmop{dy}) \otimes \sigma K^{\partial}_+ \]
with
\begin{eqnarray*}
  \Eta^{10}_t (x, y) & = & - \frac{1}{2} H_N (t, x, y) - \frac{1}{4}
  \partial_x \left( \left( \text{d}^{\tmop{GF}} \otimes 1 + 1 \otimes
  \mathd^{\tmop{GF}} \right) \int_0^t \tmop{ds} (\widetilde{H_s} - H_s)
  \right)\\
  \Eta^{01}_t (x, y) & = & \frac{1}{2} H_D (t, x, y) - \frac{1}{4} \partial_y
  \left( \left( \text{d}^{\tmop{GF}} \otimes 1 + 1 \otimes \mathd^{\tmop{GF}}
  \right) \int_0^t \tmop{ds} (\widetilde{H_s} - H_s) \right) .
\end{eqnarray*}
Now we will work out the explicit formula of $b^{\tmop{int}}$ for the examples
in the last section.

\begin{example}
  \label{egeftccltn44}For the theory given in Example \ref{eg36thycnstctn}, we
  have
  \[ e^{\hbar \partial_{P (\varepsilon, t)}} \rho (I^{\partial}) = \rho
     (I^{\partial}) = - \int_{\mathbb{R}_{\geqslant 0}} \omega^{\partial}
     (Q^{\tmop{rel}} -, -), \]
  so $b^{\tmop{int}} = p e^{\hbar (\Delta_t K^{\tmop{sym}})_2} \{ \rho
  (I^{\partial}), - \}_t i$.
  
  By construction (recall $\mathcal{E}_L =\mathcal{L} \oplus \mathcal{L}'$),
  \[ p^{- 1} (\mathcal{O} (L)) \subset \widehat{\tmop{Sym}}
     (\mathcal{L}^{\ast}), \qquad i (\mathcal{O} (L)) \subset
     \widehat{\tmop{Sym}} (\mathcal{L}^{\ast}) . \]
  Just like $\partial_{P (\varepsilon, t)}, \Delta_t$ and $\{ -, - \}_t$, the
  operator $(\Delta_t K^{\tmop{sym}})_2$ also pairs a linear functional factor
  depending on $\mathcal{L}$ with a linear functional factor depending on
  $\mathcal{L}'$. So we have
  \[ b^{\tmop{int}} = p e^{\hbar (\Delta_t K^{\tmop{sym}})_2} \{ \rho
     (I^{\partial}), - \}_t i = \hbar p (\Delta_t K^{\tmop{sym}})_2 \{ \rho
     (I^{\partial}), - \}_t i = \hbar p [(\Delta_t K^{\tmop{sym}})_2, \{ \rho
     (I^{\partial}), - \}_t ] i. \]
  By Proposition \ref{2t0oprtcmtderto2t0}, this is a $2$-to-$0$ operator,
  decided by
  \begin{eqnarray*}
    b^{\tmop{int}} (f_1 f_2) & = & \hbar \Delta_t ((K \{ \rho (I^{\partial}),
    i f_1 \}_t) (i f_2) + (i f_1) (K \{ \rho (I^{\partial}), i f_2 \}_t))\\
    & = & (- 1)^{| f_1 | + | f_2 |} 4 \hbar \int_{\mathbb{R}_{\geqslant 0}}
    \tmop{dx} \Eta^{10}_t (x, 0) \int_0^x \tmop{dy} \Eta^{10}_t (y, 0) f_1 f_2
    ((Q^{\tmop{rel}} \otimes 1) K^{\partial}_+)\\
    & = & (- 1)^{| f_1 | + | f_2 |} 2 \hbar \left( \int_0^{+ \infty}
    \tmop{dx} \Eta^{10}_t (x, 0) \right)^2 f_1 f_2 ((Q^{\tmop{rel}} \otimes 1)
    K^{\partial}_+)\\
    & = & (- 1)^{| f_1 | + | f_2 |} \frac{\hbar}{2} \left( \int_0^{+ \infty}
    \tmop{dx} H_N (t, x, 0) \right)^2 f_1 f_2 ((Q^{\tmop{rel}} \otimes 1)
    K^{\partial}_+)\\
    & = & \frac{(- 1)^{| f_1 | + | f_2 |} \hbar}{2} f_1 f_2 ((Q^{\tmop{rel}}
    \otimes 1) K^{\partial}_+),
  \end{eqnarray*}
  with $f_1, f_2 \in L^{\ast}$.
  
  So, we obtain an effective theory described by a differential BV algebra
  $(\mathcal{O} (L), Q^{\partial}, b^{\tmop{int}} / \hbar) .$ The effective BV
  operator $b^{\tmop{int}} / \hbar$ can be degenerate (as a pairing on
  $L^{\ast}$) for certain choices of $Q^{\tmop{rel}}$. This example reproduces
  the simplest case (i.e., the boundary manifold is a point here) of the
  second relation in {\cite[Theorem 3.4.3]{rabinovich2021factorization}}. It
  has been observed that degenerate classical field theories can arise from
  classical field theories on manifold with boundary
  {\cite{butson2016degenerate}}. We hope our homotopy transfer method can
  extend this study to quantum level in the future.
\end{example}

\begin{example}
  \label{egeftccltn45}For BF theory with B boundary condition (Example
  \ref{eg37thycnstctn}), by the unimodular condition $f^{c b}_c = 0$ and
  $K^{\partial}_+ = - \eta t^a \otimes \epsilon t_a, K^{\partial}_- = -
  \epsilon t_a \otimes \eta t^a$, it is direct to see
  \begin{equation}
    \partial_{P (\varepsilon, t)} \rho (I^{\partial}) = \frac{1}{2}
    \partial_{P (\varepsilon, t)} \int_{\mathbb{R} \geqslant 0} f^{a b}_c B^c
    \wedge A_a \wedge A_b = 0. \label{actnvnshppgtbfbcdt}
  \end{equation}
  So $e^{\hbar \partial_{P (\varepsilon, t)}} \rho (I^{\partial}) = \rho
  (I^{\partial})$. Using the same argument in Example \ref{egeftccltn44}, we
  have
  \[ b^{\tmop{int}} = p e^{\hbar (\Delta_t K^{\tmop{sym}})_2} \{ \rho
     (I^{\partial}), - \}_t i = \hbar p (\Delta_t K^{\tmop{sym}})_2 \{ \rho
     (I^{\partial}), - \}_t i. \]
  By the same reason for (\ref{actnvnshppgtbfbcdt}), $(\Delta_t
  K^{\tmop{sym}})_2 \{ \rho (I^{\partial}), i (f) \}_t = 0$ for $f \in
  L^{\ast}$. We conclude that
  \begin{itemizedot}
    \item for $n \geqslant 2$, $\forall f_1 f_2 \cdots f_n \in \tmop{Sym}^n
    (L^{\ast})$,
    \[ b^{\tmop{int}} (f_1 f_2 \cdots f_n) = \sum_{i < j} (\pm)_{\tmop{Kos}}
       b^{\tmop{int}} (f_i f_j) f_1 \ldots \widehat{f_i} \ldots \widehat{f_j}
       \ldots f_n . \]
    \item $b^{\tmop{int}} (\tmop{Sym}^{\leqslant 1} (L^{\ast}) [[\hbar]]) =
    0$, and $b^{\tmop{int}} (\tmop{Sym}^2 (L^{\ast})) \subset \hbar L^{\ast}$.
    For $f_1, f_2 \in L^{\ast}$,
    \begin{eqnarray*}
      b^{\tmop{int}} (f_1 f_2) & = & 4 \hbar f^{a b}_c B^c
      \int_{\mathbb{R}_{\geqslant 0}} \tmop{dx} \Eta^{10}_t (x, 0) \int_0^x
      \tmop{dy} \Eta^{10}_t (y, 0) \partial_{B^a} \partial_{B^b} (f_1 f_2)\\
      & = & \frac{\hbar}{2} f^{a b}_c B^c \partial_{B^a} \partial_{B^b} (f_1
      f_2)
    \end{eqnarray*}
    where $\partial_{B^a}$ is the derivation on $\mathcal{O} (L)$ decided by
    $\partial_{B^a} (B^b) = \delta_a^b$.
  \end{itemizedot}
  The isomorphism $L^{\ast} \simeq \mathfrak{g} [1]$ identifies $(\mathcal{O}
  (L) [[\hbar]], b^{\tmop{int}})$ with $(\widehat{\tmop{Sym}} (\mathfrak{g}
  [1]) [[\hbar]], \hbar d_{[-, -]})$, where $d_{[-, -]}$ is the differential
  of the Chevalley-Eilenberg complex
  \[ C_{\bullet} (\mathfrak{g}) \assign (\widehat{\tmop{Sym}} (\mathfrak{g}
     [1]), d_{[-, -]}) . \]
  So, we have refined the first conclusion in {\cite[Theorem
  5.0.2]{rabinovich2021factorization}} by preserving the $\hbar$-grading
  structure.
\end{example}

\begin{example}
  \label{egeftccltn46}For the theory given in Example \ref{eg38thycnstctn},
  using the same argument in Example \ref{egeftccltn44}, we have
  \[ b^{\tmop{int}} = \lim_{\varepsilon \rightarrow 0} p e^{\hbar (\Delta_t
     K^{\tmop{sym}})_2} \{ e^{\hbar \partial_{P (\varepsilon, t)}} \rho
     (I^{\partial}), - \}_t i = \frac{\hbar^2}{2} p ((\Delta_t
     K^{\tmop{sym}})_2)^2 \{ \rho (I^{\partial}), - \}_t i \]
  with $\rho (I^{\partial}) = - \int_{\mathbb{R}_{\geqslant 0}} p_1 \wedge p_1
  \wedge p_2$. This is a ``$3$-to-$0$ operator'' on $\mathcal{O} (L)
  [[\hbar]]$, in the sense that
  \begin{itemizedot}
    \item $b^{\tmop{int}} (\tmop{Sym}^{\leqslant 2} (L^{\ast}) [[\hbar]]) =
    0$, and $b^{\tmop{int}} (\tmop{Sym}^3 (L^{\ast})) \subset \mathbb{R}
    \hbar^2$.
    
    \item For $n \geqslant 3$, $\forall f_1 f_2 \cdots f_n \in \tmop{Sym}^n
    (L^{\ast})$,
    \[ b^{\tmop{int}} (f_1 f_2 \cdots f_n) = \sum_{i < j < k}
       (\pm)_{\tmop{Kos}} b^{\tmop{int}} (f_i f_j f_k) f_1 \ldots
       \widehat{f_i} \ldots \widehat{f_j} \ldots \widehat{f_k} \ldots f_n . \]
  \end{itemizedot}
  And for $f_1, f_2, f_3 \in L^{\ast}$,
  \begin{eqnarray*}
    b^{\tmop{int}} (f_1 f_2 f_3) & = & 24 \hbar^2 \int_{\mathbb{R}_{\geqslant
    0}} \tmop{dx} \Eta^{10}_t (x, 0) \int_0^x \tmop{dy} \Eta^{10}_t (y, 0)
    \int_0^x \tmop{dz} \Eta^{10}_t (z, 0) \partial_{q_1} \partial_{q_1}
    \partial_{q_2} (f_1 f_2 f_3)\\
    & = & \hbar^2 \partial_{q_1} \partial_{q_1} \partial_{q_2} (f_1 f_2 f_3)
  \end{eqnarray*}
  where $\partial_{q_i}$ is the derivation on $\mathcal{O} (L)$ decided by
  $\partial_{q_i} (q_j) = \delta_{i j}$ for $i, j \in \{ 1, 2 \} .$ Obviously,
  $b^{\tmop{int}} / \hbar^2$ is not a BV operator on $\mathcal{O} (L)$.
\end{example}

\subsection{A quasi-example: $2$D BF
theory}\label{chtngeg2dbfthyabcdtsctn42haha}

In {\cite{rabinovich2021factorization}}, field theory topological normal to
the boundary can be defined on a higher-dimensional spacetime $X$, inducing
QFT on $\mathbb{R}_{\geqslant 0} \times \partial X$. We can ask whether
Theorem \ref{mainthmtqmefthysdr} can be extended to the case when $\partial X$
is not a point. Now, to perform renormalization, the propagator will contain
``$\mathbb{R}_{\geqslant 0}$ direction'' terms and ``$\partial X$ direction''
terms. This fact unfortunately invalidates the de Rham degree argument in the
proof of Theorem \ref{mainthmtqmefthysdr}. So, in the following we do not
define the theory rigorously, but instead, make a trial regarding the theory
on $\mathbb{R}_{\geqslant 0} \times \partial X$ as a TQM on
$\mathbb{R}_{\geqslant 0}$ (with singularity in the data). In this way we will
study BF theory on $\mathbb{R}_{\geqslant 0} \times \mathbb{R}$, and see
whether we can obtain some meaningful ``effective theory'' by applying Theorem
\ref{mainthmtqmefthysdr} anyway.

\subsubsection*{B boundary condition}

Let $\mathfrak{g}$ be the unimodular Lie algebra mentioned in Example
\ref{eg37thycnstctn}. For BF theory on $\mathbb{R}_{\geqslant 0} \times
\mathbb{R}$ with B boundary condition, we define
\[ L \assign \Omega^{\bullet} (\mathbb{R}) \otimes \mathfrak{g}^{\ast}, \qquad
   L' \assign \Omega^{\bullet} (\mathbb{R}) \otimes \mathfrak{g} [1], \qquad V
   \assign L \oplus L' . \]
Let $V_c$ denote the set of compactly supported forms in $V$. There is a
degree $0$ nondegenerate antisymmetric pairing $\omega^{\partial}$ on $V_c$,
decided by
\[ \omega^{\partial} (\varphi_1 \otimes t_a, \varphi_2 \otimes \eta t^b)
   \assign \delta_a^b \int_{\mathbb{R}} \varphi_1 \wedge \varphi_2 \]
with $\varphi_1, \varphi_2 \in \Omega^{\bullet} (\mathbb{R})$ having compact
supports and $\eta$ being formal variable of degree $- 1$. Let $Q^{\partial}$
be the de Rham operator on $\Omega^{\bullet} (\mathbb{R})$, then,
$Q^{\partial} \omega^{\partial} = 0$. Now,
\[ \mathcal{E}_L = \{ f \in \Omega^{\bullet} (\mathbb{R}_{\geqslant 0})
   \otimes V| \iota^{\ast} f \in L \}, \]
with $\iota^{\ast}$ induced by $\iota : \mathbb{R} \hookrightarrow
\mathbb{R}_{\geqslant 0} \times \mathbb{R}$. $\mathcal{O} (\mathcal{E}_L)$
consists of functionals compactly supported on $\mathbb{R}_{\geqslant 0}
\times \mathbb{R}$. We can write
\[ K^{\partial}_+ = \delta (x' - y') (\tmop{dy}' - \tmop{dx}') \otimes \eta
   t^a \otimes t_a, \qquad K^{\partial}_- = - \sigma K^{\partial}_+ \]
where $(x', y')$ are coordinates on $\mathbb{R} \times \mathbb{R}$. Then, we
formally define $K_t$ and $P (\varepsilon, \Lambda)$ as in
(\ref{bvknldfntn311}) and (\ref{ppgtdfntn31212}). Let $\mathd$ still denote
the de Rham operator on $\Omega^{\bullet} (\mathbb{R}_{\geqslant 0})$. The
data
\[ (\mathcal{O} (\mathcal{E}_L), Q = \mathd + Q^{\partial}, \Delta_t =
   \partial_{K_t}) \]
is ill-defined, but if we restrict $\mathcal{O} (\mathcal{E}_L)$ to its
subspace of smooth distributions, it will become a differential BV algebra.
Let $\{ \Beta^a \}_{a = 1}^{\ell}$ be the basis of
$(\mathfrak{g}^{\ast})^{\ast}$ dual to $\{ t_a \}_{a = 1}^{\ell}$, $\{ A_a
\}_{a = 1}^{\ell}$ be the basis of $(\mathfrak{g} [1])^{\ast}$ such that $A_a
(\eta t^b) = \delta_a^b$, and $I^{\partial} \assign \frac{1}{2}
\int_{\mathbb{R}} f^{a b}_c \Beta^c \wedge A_a \wedge A_b$. Then we pretend
that
\[ \rho (I^{\partial}) = \frac{1}{2} \int_{\mathbb{R}_{\geqslant 0}}
   \int_{\mathbb{R}} f^{a b}_c \Beta^c \wedge A_a \wedge A_b, \]
together with $(\mathcal{O} (\mathcal{E}_L), Q, \Delta_t)$, defines an
interactive theory in the sense of Definition \ref{tqmwobdrcdtn35}. The proof
of Theorem \ref{mainthmtqmefthysdr} and arguments in Example
\ref{egeftccltn45} formally hold, allowing us to tentatively write down an
effective observable complex $(\mathcal{O} (L) [[\hbar]] =\mathcal{O}
(\Omega^{\bullet} (\mathbb{R}) \otimes \mathfrak{g}^{\ast}) [[\hbar]],
Q^{\partial} + b^{\tmop{int}})$, where
\begin{itemizedot}
  \item for $n \geqslant 2$, $\forall f_1 f_2 \cdots f_n \in \tmop{Sym}^n
  (L^{\ast})$,
  \[ b^{\tmop{int}} (f_1 f_2 \cdots f_n) = \sum_{i < j} (\pm)_{\tmop{Kos}}
     b^{\tmop{int}} (f_i f_j) f_1 \ldots \widehat{f_i} \ldots \widehat{f_j}
     \ldots f_n, \]
  \item $b^{\tmop{int}} (\tmop{Sym}^{\leqslant 1} (L^{\ast}) [[\hbar]]) = 0$,
  and $b^{\tmop{int}} (\tmop{Sym}^2 (L^{\ast})) \subset \hbar L^{\ast}$. For
  $f_1, f_2 \in L^{\ast}$, suppose
  \[ f_j = \int_{\mathbb{R}} f_{j a} \wedge \Beta^a \quad \text{for } j = 1, 2
  \]
  with each $\int_{\mathbb{R}} f_{j a} \wedge (-)$ representing a functional
  on $\Omega^{\bullet} (\mathbb{R})$, then
  \[ b^{\tmop{int}} (f_1 f_2) = \hbar (- 1)^{| f_{1 a} | + 1} f^{a b}_c
     \int_{\mathbb{R}} f_{1 a} \wedge f_{2 b} \wedge \Beta^c . \]
\end{itemizedot}
This is still ill-defined, because we cannot multiply two distributions.
Recall that we have an embedding
\begin{equation}
  (\Omega^{\bullet}_c (\mathbb{R}) [1], - \mathd_{\mathbb{R}}) \mapsto
  ((\Omega^{\bullet} (\mathbb{R}))^{\ast}, Q^{\partial}) \qquad \eta f \mapsto
  \int_{\mathbb{R}} f \wedge (-), \label{atyhbttlmmsctn411}
\end{equation}
with $(\Omega^{\bullet}_c (\mathbb{R}), \mathd_{\mathbb{R}})$ being the
compactly supported de Rham complex on $\mathbb{R}$ and $\eta$ being a degree
$- 1$ formal variable. By Atiyah-Bott lemma, this is a continuous
quasi-isomorphism (we refer to {\cite[Appendix D]{costello_gwilliam_2016}} for
an explanation). (\ref{atyhbttlmmsctn411}) and the identification
$(\mathfrak{g}^{\ast})^{\ast} \simeq \mathfrak{g}$ induce a quasi-isomorphism
\[ (\widehat{\tmop{Sym}} (\Omega^{\bullet}_c (\mathbb{R}) [1] \otimes
   \mathfrak{g}), - \mathd_{\mathbb{R}}) \hookrightarrow (\mathcal{O} (L),
   Q^{\partial}) . \]
Then, $- \mathd_{\mathbb{R}} + b^{\tmop{int}}$ is a well-defined differential
on $\widehat{\tmop{Sym}} (\Omega^{\bullet}_c (\mathbb{R}) [1] \otimes
\mathfrak{g}) [[\hbar]]$. For $f_1, f_2 \in \Omega^{\bullet}_c (\mathbb{R})
\otimes \mathfrak{g}$,
\[ b^{\tmop{int}} ((\eta f_1) (\eta f_2)) = \hbar (- 1)^{| f_1 | + 1} \eta
   [f_1, f_2], \]
where we have recognized $\Omega^{\bullet}_c (\mathbb{R}) \otimes
\mathfrak{g}$ to be a dg Lie algebra, with bracket still denoted by $[-, -]$.
In this way, we can identify
\begin{equation}
  (\widehat{\tmop{Sym}} (\Omega^{\bullet}_c (\mathbb{R}) [1] \otimes
  \mathfrak{g}) [[\hbar]], - \mathd_{\mathbb{R}} + b^{\tmop{int}})
  \label{fakecalcuresult2dbfcdtnb411}
\end{equation}
with the Chevalley-Eilenberg complex $C_{\bullet} (\Omega^{\bullet}_c
(\mathbb{R}) \otimes \mathfrak{g})$ (after taking $\hbar = 1$).

\begin{remark}
  If we take the boundary manifold to be an open subset $U \subset
  \mathbb{R}$, the above quasi-calculation will give rise to $C_{\bullet}
  (\Omega^{\bullet}_c (U) \otimes \mathfrak{g})$. This assignment actually
  leads to the enveloping factorization algebra $\mathbb{U}
  (\Omega^{\bullet}_{\mathbb{R}} \otimes \mathfrak{g})$ on $\mathbb{R}$ (also
  called factorization envelope, see {\cite[Section
  3.6]{costello_gwilliam_2016}}). The cohomological factorization algebra of
  $\mathbb{U} (\Omega^{\bullet}_{\mathbb{R}} \otimes \mathfrak{g})$ is the
  universal enveloping algebra $U\mathfrak{g}$, regarded as a factorization
  algebra on $\mathbb{R}$ (see {\cite[Section 3.4]{costello_gwilliam_2016}}).
\end{remark}

\subsubsection*{A boundary condition}

For the A boundary condition, we modify the defining data in the above case as
follows:
\[ L = \Omega^{\bullet} (\mathbb{R}) \otimes \mathfrak{g} [1], \qquad L' =
   \Omega^{\bullet} (\mathbb{R}) \otimes \mathfrak{g}^{\ast}, \qquad
   K^{\partial}_+ = \delta (x' - y') (\tmop{dy}' - \tmop{dx}') \otimes t_a
   \otimes \eta t^a . \]
Then, $\mathcal{O} (L) [[\hbar]] =\mathcal{O} (\Omega^{\bullet} (\mathbb{R})
\otimes \mathfrak{g} [1]) [[\hbar]]$, and we can formally derive
\[ b^{\tmop{int}} = p e^{\hbar (\Delta_t K^{\tmop{sym}})_2} \{ \rho
   (I^{\partial}), - \}_t i = p \{ \rho (I^{\partial}), - \}_t i. \]
So, we obtain a cochain complex
\begin{equation}
  (\mathcal{O} (\Omega^{\bullet} (\mathbb{R}) \otimes \mathfrak{g} [1])
  [[\hbar]], \mathd_{\mathbb{R}} + b^{\tmop{int}}),
  \label{fakecalcuresult2dbfcdtna412}
\end{equation}
where $b^{\tmop{int}}$ is a derivation decided by
\[ b^{\tmop{int}} \left( \int_{\mathbb{R}} f^a \wedge A_a \right) = \frac{(-
   1)^{| f^c |}}{2} f^{a b}_c \int_{\mathbb{R}} f^c \wedge A_a \wedge A_b . \]
Thus we can identify $(\mathcal{O} (\Omega^{\bullet} (\mathbb{R}) \otimes
\mathfrak{g} [1]), \mathd_{\mathbb{R}} + b^{\tmop{int}})$ with $\tmop{CE}
(\Omega^{\bullet} (\mathbb{R}) \otimes \mathfrak{g})$, the Chevalley-Eilenberg
algebra of the dg Lie algebra $\Omega^{\bullet} (\mathbb{R}) \otimes
\mathfrak{g}$.

\begin{remark}
  If we apply the constructions in {\cite[Section
  4.7]{rabinovich2021factorization}} to BF theory on $\mathbb{R}_{\geqslant 0}
  \times \mathbb{R}$, we will obtain a factorization algebra of observables
  living on $\mathbb{R}_{\geqslant 0} \times \mathbb{R}$, denoted by
  $\tmop{Obs}^q_{\tmop{BF}, B}$ or $\tmop{Obs}^q_{\tmop{BF}, A}$ for the B or
  A boundary condition case. Then, the projection $\mathpi :
  \mathbb{R}_{\geqslant 0} \times \mathbb{R} \mapsto \mathbb{R}$ will induce
  pushforward factorization algebras $\mathpi_{\ast} (\tmop{Obs}^q_{\tmop{BF},
  B})$ and $\mathpi_{\ast} (\tmop{Obs}^q_{\tmop{BF}, A})$ on $\mathbb{R}$. If
  we figured out a rigorous version of the definitions and calculations
  presented in the current subsection, we would confirm a weak equivalence
  between $\mathpi_{\ast} (\tmop{Obs}^q_{\tmop{BF}, B}) |_{\hbar = 1}$ and
  $U\mathfrak{g}$, and a weak equivalence between $\mathpi_{\ast}
  (\tmop{Obs}^q_{\tmop{BF}, A}) |_{\hbar = 1}$ and $\tmop{CE} (\mathfrak{g})$.
  Expectations for this relation can be found in the literature (see e.g.,
  {\cite[Remark 16]{gwilliam2021factorization}}).
\end{remark}

\section{Effective Theory and Derived BV
Algebras}\label{sectn5eftdrvdbvalg220307}

If we use BV formalism to construct a renormalized QFT, the differential on
the observable complex (\ref{abstctgnrlintobscmplx001}) will always consist of
a BV operator (``second-order'') and a derivation (``first-order''). However,
we have seen in Example \ref{egeftccltn46} that the differential on the
effective observables can be a ``higher-order'' operator. In order to describe
it, we resort to a generalization of BV algebra called ``derived BV algebra'',
introduced by Olga Kravchenko {\cite{kravchenko1999deformations}} (the
terminology there is ``commutative $\tmop{BV}_{\infty}$ algebra''). In
{\cite{bandiera2020cumulants}}, Ruggero Bandiera showed that derived BV
algebra structures can be transferred along certain kind of SDR's via
homological perturbation theory. We will review this result (also simplify it
a little) in this section, and conclude that constructions in previous
sections fit into derived BV algebra structure (see Proposition
\ref{hrgdrvdbvmpsm512} and Proposition \ref{reallythelastpp519ref}).

\subsubsection*{Cumulants and Koszul brackets}

In this subsection we collect basics to prepare for discussions of derived BV
algebras. We refer to {\cite{bandiera2020cumulants}} and references therein
for details.

\begin{definition}
  \label{dfntncmlts51}Let $\mathcal{A}, \mathcal{B}$ be graded commutative
  unital algebras over field $k$, $f : \mathcal{A} \mapsto \mathcal{B}$ be a
  degree $0$ linear map such that $f (\tmmathbf{1}_{\mathcal{A}})
  {=\tmmathbf{1}_{\mathcal{B}}} $. The {\tmstrong{cumulants}} $\{ \kappa (f)_n
  : \mathcal{A}^{\otimes n} \mapsto \mathcal{B}|n \in \mathbb{Z}_+ \}$ are
  defined by the following formula:
  \begin{eqnarray}
    &  & \kappa (f)_n (a_1, a_2, \ldots, a_n) \assign \nonumber\\
    &  & \sum_{\tmscript{\begin{array}{c}
      j, m_1, \ldots, m_j \geqslant 1,\\
      m_1 + \cdots + m_j = n
    \end{array}}} \sum_{\sigma \in \mathbf{S} (m_1, \ldots, m_j)} \frac{(-
    1)^{j - 1}}{j}  (\pm)_{\tmop{Kos}} f (a_{\sigma (1)} \cdots a_{\sigma
    (m_1)}) \cdots f (a_{\sigma (n - m_j + 1)} \cdots a_{\sigma (n)}) 
    \label{hophopdeflastfmldef515151}
  \end{eqnarray}
  where $a_1, a_2, \ldots, a_n \in \mathcal{A}$, and
  \begin{equation}
    \mathbf{S} (m_1, \ldots, m_j) \assign \left\{ \sigma \in \mathbf{S}_n |
    \sigma (1) < \cdots < \sigma (m_1), \quad \ldots \text{} \hspace{1.0em},
    \sigma (n - m_j + 1) < \cdots < \sigma (n) \right\} \label{icbrpttn52}
  \end{equation}
  can be regarded as the set of ordered $(m_1, \ldots, m_j)$-partitions of $\{
  1, 2, \ldots, n \}$. Actually $\mathbf{S} (m_1, \ldots, m_j)$ is defined for
  $m_1, \ldots, m_j \geqslant 0$, although in
  (\ref{hophopdeflastfmldef515151}) we only sum over $\sigma \in \mathbf{S}
  (m_1, \ldots, m_j)$ with $m_1, \ldots, m_j \geqslant 1$.
\end{definition}

By definition, $\kappa (f)_n$ is symmetric in the inputs, and $\kappa (f)_1 =
f$. Roughly speaking, the cumulants of $f$ for $n \geqslant 2$ measure the
deviation of $f$ from being a morphism of algebras. For example, it is direct
to see
\[ \kappa (f)_2 (a_1, a_2) = f (a_1 a_2) - f (a_1) f (a_2) . \]
If $f$ is a morphism of algebras, then $\kappa (f)_n = 0$ for all $n \geqslant
2$.

\begin{definition}
  \label{dfntnkzlbrkt52}Let $\mathcal{A}$ be a graded commutative unital
  algebra over field $k$, $D : \mathcal{A} \mapsto \mathcal{A}$ be a linear
  map such that $D (\tmmathbf{1}_{\mathcal{A}}) = 0$. The {\tmstrong{Koszul
  brackets}} $\{ \mathcal{K} (D)_n : \mathcal{A}^{\otimes n} \mapsto
  \mathcal{A}|n \in \mathbb{Z}_+ \}$ are defined by the following formula:
  \[ \mathcal{K} (D)_n (a_1, a_2, \ldots, a_n) \assign \sum_{m = 1}^n
     \sum_{\sigma \in \mathbf{S} (m, n - m)} (\pm)_{\tmop{Kos}} (- 1)^{n - m}
     D (a_{\sigma (1)} \cdots a_{\sigma (m)}) a_{\sigma (m + 1)} \cdots
     a_{\sigma (n)} \]
  for $a_1, a_2, \ldots, a_n \in \mathcal{A}$.
\end{definition}

By definition, $\mathcal{K} (D)_n$ is symmetric in the inputs, and
$\mathcal{K} (D)_1 = D$. Roughly speaking, the Koszul brackets of $D$ for $n
\geqslant 2$ measure the deviation of $D$ from being a derivation of graded
algebras. For example, it is direct to see
\[ \mathcal{K} (D)_2 (a_1, a_2) = D (a_1 a_2) - D (a_1) a_2 - (- 1)^{| D | |
   a_1 |} a_1 D (a_2) . \]
\begin{remark}
  For $n \geqslant 1$, let $d_1, d_2, \ldots, d_n$ be derivations on
  $\mathcal{A}$. Then, their composition $D \assign d_1 d_2 \ldots d_n$ is an
  ``$n$-th-order'' operator, satisfying the following properties:
  \begin{itemizedot}
    \item $\mathcal{K} (D)_m = 0$ for $\forall m > n$;
    
    \item $\mathcal{K} (D)_n$ is a ``multi-derivation'', i.e., it is a
    derivation with respect to each input argument.
  \end{itemizedot}
  In Definition \ref{dfntndgbvalgbrasec2229}, we require the BV operator
  $\Delta$ to be ``second-order'', so $\mathcal{K} (\Delta)_{m > 2} = 0$. The
  BV bracket $\{ -, - \}$ is exactly $\mathcal{K} (\Delta)_2$.
\end{remark}

Besides defining $\{ \kappa (f)_n \}$ and $\{ \mathcal{K} (D)_n \}$ by
concrete formulae as above, we can also generate these maps from $f$ and $D$
in a systematic way. This method involves the cofree cocommutative coalgebra
cogenerated by $\mathcal{A}, \mathcal{B}$.

Given a graded vector space $V$ over $k$, we can construct a coalgebra
$(\tmop{Sym} (V), \tmmathbf{\Delta}, \tmmathbf{\varepsilon})$, where
\[ \tmmathbf{\Delta} (v_1 v_2 \cdots v_n) \assign \sum_{m = 0}^n \sum_{\sigma
   \in \mathbf{S} (m, n - m)} (\pm)_{\tmop{Kos}} (v_{\sigma (1)} v_{\sigma
   (2)} \cdots v_{\sigma (m)}) \otimes (v_{\sigma (m + 1)} v_{\sigma (m + 2)}
   \cdots v_{\sigma (n)}) \]
for $v_1, v_2, \ldots, v_n \in V$, and $\tmmathbf{\varepsilon}$ is the
projection $\tmop{Sym} (V) \mapsto \tmop{Sym}^0 (V)$.

There is a correspondence:
\begin{equation}
  \left\{ \begin{array}{c}
    \text{coalgebra morphism}\\
    F : \tmop{Sym} (V) \mapsto \tmop{Sym} (W)
  \end{array} \right\} \simeq \left\{ \begin{array}{c}
    \text{linear maps } (f_1, \ldots, f_n, \ldots),\\
    f_n \in \tmop{Hom}_k (\tmop{Sym}^n (V), W)
  \end{array} \right\} . \label{crspdccoalgmor}
\end{equation}
The bijections between them are given by:
\begin{eqnarray*}
  F (v_1 \cdots v_n) & = & \sum_{\tmscript{\begin{array}{c}
    j, m_1, \ldots, m_j \geqslant 1,\\
    m_1 + \cdots + m_j = n
  \end{array}}} \frac{1}{j!} \sum_{\sigma \in \mathbf{S} (m_1, \ldots, m_j)}
  (\pm)_{\tmop{Kos}} f_{m_1} (v_{\sigma (1)} \cdots v_{\sigma (m_1)}) \cdots
  f_{m_j} (v_{\sigma (n - m_j + 1)} \cdots v_{\sigma (n)}),\\
  f_n & = & \mathbf{p}_1 F\mathbf{i}_n
\end{eqnarray*}
where $\mathbf{i}_n, \mathbf{p}_n$ are the natural injection and projection
between $\tmop{Sym}^n (-)$ and $\tmop{Sym} (-)$. The linear maps $(f_1,
\ldots, f_n, \ldots)$ are called the {\tmstrong{Taylor coefficients}} of $F$.
In a similar manner, we have another correspondence:
\begin{equation}
  \left\{ \begin{array}{c}
    \text{coderivation } Q \text{ on } \tmop{Sym} (V)
  \end{array} \right\} \simeq \left\{ \begin{array}{c}
    \text{linear maps } (q_0, q_1, \ldots, q_n, \ldots),\\
    q_n \in \tmop{Hom}_k (\tmop{Sym}^n (V), V)
  \end{array} \right\} \label{crspdnccodrvtn}
\end{equation}
with the bijections between them given by:
\begin{eqnarray*}
  Q (v_1 \cdots v_n) & = & \sum_{m = 0}^n \sum_{\sigma \in \mathbf{S} (m, n -
  m)} (\pm)_{\tmop{Kos}} q_m (v_{\sigma (1)} \cdots v_{\sigma (m)}) v_{\sigma
  (m + 1)} \cdots v_{\sigma (n)},\\
  q_n & = & \mathbf{p}_1 Q\mathbf{i}_n .
\end{eqnarray*}
Now, $(q_0, q_1, \ldots, q_n, \ldots)$ are called the Taylor coefficients of
$Q$.

\begin{remark*}
  To avoid confusion in the rest of this paper, for an algebra $\mathcal{A}$,
  we will use $S (\mathcal{A})$ to denote $\tmop{Sym} (\mathcal{A})$, and use
  $\odot$ to denote the symmetric tensor product. So, $\tmop{Sym}^n
  (\mathcal{A})$ will be written as $\mathcal{A}^{\odot n}$. The
  multiplication mark of $\mathcal{A}$ itself will still be omitted.
\end{remark*}

Let $\mathcal{A}, \mathcal{B}$ be graded commutative unital algebras. For $t
\in k, t \neq 0$, we have a coalgebra automorphism
\begin{equation}
  \mathbf{E}^{\mathcal{A}}_t : S (\mathcal{A}) \mapsto S (\mathcal{A}), \qquad
  \mathbf{p}_1 \mathbf{E}^{\mathcal{A}}_t \mathbf{i}_n (a_1 \odot a_2 \odot
  \cdots \odot a_n) = \frac{1}{t^{n - 1}} a_1 a_2 \cdots a_n .
  \label{coalgautomexpntl}
\end{equation}
It is direct to verify that the inverse of $\mathbf{E}^{\mathcal{A}}_t$ is
\begin{equation}
  \mathbf{L}^{\mathcal{A}}_t : S (\mathcal{A}) \mapsto S (\mathcal{A}), \qquad
  \mathbf{p}_1 \mathbf{L}^{\mathcal{A}}_t \mathbf{i}_n (a_1 \odot a_2 \odot
  \cdots \odot a_n) = \frac{(- 1)^{n - 1} (n - 1) !}{t^{n - 1}} a_1 a_2 \cdots
  a_n . \label{coalgautologrthm}
\end{equation}
Given a degree $0$ linear map $f : \mathcal{A} \mapsto \mathcal{B}$ satisfying
$f (\tmmathbf{1}_{\mathcal{A}}) {=\tmmathbf{1}_{\mathcal{B}}} $, we define a
coalgebra morphism
\[ S (f) : S (\mathcal{A}) \mapsto S (\mathcal{B}), \qquad \text{its Taylor
   coefficients are } (f, 0, \ldots, 0, \ldots) . \]
\begin{proposition}
  \label{rltnbtwncmltsandcmltscoalgmor53}In above settings, we define a
  coalgebra morphism
  \begin{equation}
    \tmmathbf{\kappa}_t (f) : S (\mathcal{A}) \mapsto S (\mathcal{B}), \qquad
    \tmmathbf{\kappa}_t (f) \assign \mathbf{L}^{\mathcal{B}}_t S (f)
    \mathbf{E}^{\mathcal{A}}_t . \label{coalgdfntforcmltssec5}
  \end{equation}
  Then, the cumulants in Definition \ref{dfntncmlts51} can be obtained from
  $\tmmathbf{\kappa}_t (f)$:
  \[ \kappa (f)_n (a_1, a_2, \ldots, a_n) = t^{n - 1} \mathbf{p}_1
     \tmmathbf{\kappa}_t (f) \mathbf{i}_n (a_1 \odot a_2 \odot \cdots \odot
     a_n) . \]
\end{proposition}

Given a linear map $D : \mathcal{A} \mapsto \mathcal{A}$ satisfying $D
(\tmmathbf{1}_{\mathcal{A}}) = 0$, we define a coderivation
\[ \overline{D} : S (\mathcal{A}) \mapsto S (\mathcal{A}), \qquad \text{its
   Taylor coefficients are } (0, D, 0, \ldots, 0, \ldots) . \]
\begin{proposition}
  \label{rltnbetkszlbraandkszlcodercstctn}In above settings, we define a
  coderivation
  \begin{equation}
    \tmmathbf{\mathcal{K}}_t (D) : S (\mathcal{A}) \mapsto S (\mathcal{A}),
    \qquad \tmmathbf{\mathcal{K}}_t (D) \assign \mathbf{L}^{\mathcal{A}}_t 
    \overline{D} \mathbf{E}^{\mathcal{A}}_t .
    \label{coalgdfntnforkzlbrcktsec5}
  \end{equation}
  Then, $\mathbf{p}_1 \tmmathbf{\mathcal{K}}_t (D) \mathbf{i}_0 = 0$, and the
  Koszul brackets in Definition \ref{dfntnkzlbrkt52} can be obtained from
  $\tmmathbf{\mathcal{K}}_t (D)$:
  \[ \mathcal{K} (D)_n (a_1, a_2, \ldots, a_n) = t^{n - 1} \mathbf{p}_1
     \tmmathbf{\mathcal{K}}_t (D) \mathbf{i}_n (a_1 \odot a_2 \odot \cdots
     \odot a_n) . \]
\end{proposition}

These two propositions can be verified by direct computation. Moreover,
$\tmmathbf{\kappa}_t$ and $\tmmathbf{\mathcal{K}}_t$ satisfy the following
relations:
\begin{itemizedot}
  \item For a degree $0$ map $D : \mathcal{A} \mapsto \mathcal{A}$ satisfying
  $D (\tmmathbf{1}_{\mathcal{A}}) = 0$, suppose the map $e^D$ is well defined
  on $\mathcal{A}$. Then $e^D (\tmmathbf{1}_{\mathcal{A}})
  =\tmmathbf{1}_{\mathcal{A}}$, and $e^{\tmmathbf{\mathcal{K}}_t (D)}$ is a
  well-defined automorphism on $S (\mathcal{A})$, satisfying
  \begin{equation}
    e^{\tmmathbf{\mathcal{K}}_t (D)} =\tmmathbf{\kappa}_t (e^D) .
    \label{expntlrltnkK55}
  \end{equation}
  \item Let $f : \mathcal{A} \mapsto \mathcal{B}, g : \mathcal{B} \mapsto
  \mathcal{C}$ be degree $0$ linear maps such that $f
  (\tmmathbf{1}_{\mathcal{A}}) =\tmmathbf{1}_{\mathcal{B}}, g
  (\tmmathbf{1}_{\mathcal{B}}) {=\tmmathbf{1}_{\mathcal{C}}} $, then
  \begin{equation}
    \tmmathbf{\kappa}_t (g f) =\tmmathbf{\kappa}_t (g) \tmmathbf{\kappa}_t (f)
    . \label{kassociatveppt56}
  \end{equation}
\end{itemizedot}

\subsubsection*{Derived BV algebras}

Now we are ready to introduce the derived BV algebras.

\begin{definition}
  \label{20220309lbl1}Let $\mathcal{A}$ be a graded commutative unital algebra
  over $k$, $\hbar$ be a formal variable with $| \hbar |$ being an even
  integer. Let $D$ be a $k [[\hbar]]$-linear operator on $\mathcal{A}
  [[\hbar]]$, $| D | = 1, D (\tmmathbf{1}_{\mathcal{A}}) = 0$. Then,
  $(\mathcal{A} [[\hbar]], D)$ is called a degree $(1 - | \hbar |)$ derived BV
  algebra if the following conditions are satisfied:
  \begin{itemize}
    \item $D^2 = 0$;
    
    \item if we $k ((\hbar))$-linearly extend $D$ to an operator on
    $\mathcal{A} ((\hbar))$, we can define $\tmmathbf{\mathcal{K}}_{\hbar}
    (D)$ on $S (\mathcal{A}) ((\hbar))$ as in
    (\ref{coalgdfntnforkzlbrcktsec5}), then $\tmmathbf{\mathcal{K}}_{\hbar}
    (D)$ preserves $S (\mathcal{A}) [[\hbar]] \subset S (\mathcal{A})
    ((\hbar))$.
  \end{itemize}
\end{definition}

\begin{remark}
  By Propositioin \ref{rltnbetkszlbraandkszlcodercstctn}, the second condition
  is equivalent to
  \begin{equation}
    \mathcal{K} (D)_n (a_1, \ldots, a_n) \equiv 0 \quad \left( \text{mod }
    \hbar^{n - 1} \right) \qquad \text{for all } n \geqslant 2, \text{ and }
    a_1, \ldots, a_n \in \mathcal{A}. \label{eqvdfntnforbvdrvdmdl}
  \end{equation}
  Hence Definition \ref{20220309lbl1} is equivalent to {\cite[Definition
  2.1]{bandiera2020cumulants}}. If we expand $D$ as
  \[ D = \sum_{n = 0}^{+ \infty} \hbar^n D_n, \qquad D_n \in \tmop{Hom}_k
     (\mathcal{A}, \mathcal{A}), \]
  then (\ref{eqvdfntnforbvdrvdmdl}) is equivalent to $\mathcal{K} (D_n)_{n +
  2} = 0$ for all $n \geqslant 0$.
\end{remark}

In this paper we will take $| \hbar | = 0$ if the degree is not explicitly
specified.

\begin{remark}
  \label{rmkdrvdbvalgefteg58}The observable complex
  (\ref{abstctgnrlintobscmplx001}) is born to be a derived BV algebra.
  Besides, the effective observable complexes in Proposition
  \ref{freethyeffcorigsdrrslt28}, \ref{intactvthyclsdmfdeffcorigsdrrslt217},
  and those $(\mathcal{O} (L) [[\hbar]], Q^{\partial} + b^{\tmop{int}})$'s in
  Example \ref{egeftccltn44}, \ref{egeftccltn45}, \ref{egeftccltn46}, and the
  quasi-calculation results (\ref{fakecalcuresult2dbfcdtnb411}),
  (\ref{fakecalcuresult2dbfcdtna412}) are all derived BV algebras.
\end{remark}

\begin{definition}
  \label{20220309lbl2}Given a pair of degree $(1 - | \hbar |)$ derived BV
  algebras $(\mathcal{A} [[\hbar]], D), (\mathcal{B} [[\hbar]], D')$, a
  morphism between them is a degree $0$ $k [[\hbar]]$-linear map $f :
  \mathcal{A} [[\hbar]] \mapsto \mathcal{B} [[\hbar]], f
  (\tmmathbf{1}_{\mathcal{A}}) =\tmmathbf{1}_{\mathcal{B}}$, satisfying:
  \begin{itemizedot}
    \item $f D = D' f$;
    
    \item if we $k ((\hbar))$-linearly extend $f$ to $f : \mathcal{A}
    ((\hbar)) \mapsto \mathcal{B} ((\hbar))$, we can define
    \[ \tmmathbf{\kappa}_{\hbar} (f) : S (\mathcal{A}) ((\hbar)) \mapsto S
       (\mathcal{B}) ((\hbar)) \]
    as in (\ref{coalgdfntforcmltssec5}), then $\tmmathbf{\kappa}_{\hbar} (f)
    (S (\mathcal{A}) [[\hbar]]) \subset S (\mathcal{B}) [[\hbar]]$.
  \end{itemizedot}
  By the relation (\ref{kassociatveppt56}), compositions of such morphisms are
  well defined, so we obtain a category of derived BV algebras.
\end{definition}

\begin{remark}
  By Proposition \ref{rltnbtwncmltsandcmltscoalgmor53}, the second condition
  is equivalent to
  \[ \kappa (f)_n (a_1, a_2, \ldots, a_n) \equiv 0 \quad \left( \text{mod }
     \hbar^{n - 1} \right) \qquad \text{for all } n \geqslant 2, \text{ and }
     a_1, \ldots, a_n \in \mathcal{A}. \]
  Hence Definition \ref{20220309lbl2} is equivalent to {\cite[Definition
  2.7]{bandiera2020cumulants}}.
\end{remark}

\begin{remark}
  The map $i_{\hbar} = i$ in (\ref{freeqtmeffthyreslt}) is a morphism of
  algebras, hence $\kappa (i)_n = 0$ for $n \geqslant 2$. So, $i_{\hbar}$ is a
  morphism of derived BV algebras. As for $p_{\hbar} = p e^{\hbar (\Delta
  K^{\tmop{sym}})_2}$ in (\ref{freeqtmeffthyreslt}), by (\ref{expntlrltnkK55})
  and (\ref{kassociatveppt56}) we have
  \[ \tmmathbf{\kappa}_{\hbar} (p_{\hbar}) =\tmmathbf{\kappa}_{\hbar} (p)
     \tmmathbf{\kappa}_{\hbar} (e^{\hbar (\Delta K^{\tmop{sym}})_2})
     =\tmmathbf{\kappa}_{\hbar} (p) e^{\tmmathbf{\mathcal{K}}_t (\hbar (\Delta
     K^{\tmop{sym}})_2)} . \]
  Hence $p_{\hbar}$ is also a morphism of derived BV algebras.
\end{remark}

\begin{proposition}
  \label{hrgdrvdbvmpsm512}The second HRG operators $\mathcal{W} (\partial_{P
  (\varepsilon, \Lambda)}, I_{\varepsilon}, -)$ and $\mathcal{W} (\partial_{P
  (\Lambda, \varepsilon)}, I_{\Lambda}, -)$ in
  (\ref{cjgtnofqtmobsvblcplxopnclsdcase3121}) are morphisms of derived BV
  algebras.
\end{proposition}

\begin{proof}
  Obviously $\mathcal{W} (\partial_{P (\varepsilon, \Lambda)},
  I_{\varepsilon}, -)$ preserves the unit. By (\ref{2ndhrgasspptysec3319}) and
  (\ref{kassociatveppt56}), we have
  \[ \tmmathbf{\kappa}_{\hbar} (\mathcal{W} (\partial_{P (\varepsilon, \Lambda
     + t)}, I_{\varepsilon}, -)) =\tmmathbf{\kappa}_{\hbar} (\mathcal{W}
     (\partial_{P (\Lambda, \Lambda + t)}, I_{\Lambda}, -))
     \tmmathbf{\kappa}_{\hbar} (\mathcal{W} (\partial_{P (\varepsilon,
     \Lambda)}, I_{\varepsilon}, -)) . \]
  Recall that $P (\Lambda, \Lambda) = 0$ and $\mathcal{W} (0, I, -) =
  \tmop{id}_{\mathcal{O} (\mathcal{E}_L) [[\hbar]]}$, we have
  \[ \left(  \small{\frac{\partial}{\partial \Lambda'}} S (\mathcal{W}
     (\partial_{P (\Lambda, \Lambda')}, I_{\Lambda}, -)) \right)_{\Lambda' =
     \Lambda} = \overline{\left( \small{\frac{\partial}{\partial \Lambda'}}
     \mathcal{W} (\partial_{P (\Lambda, \Lambda')}, I_{\Lambda}, -)
     \right)_{\Lambda' = \Lambda}} . \]
  So,
  \begin{eqnarray*}
    &  & \small{\frac{\partial}{\partial \Lambda}} \tmmathbf{\kappa}_{\hbar}
    (\mathcal{W} (\partial_{P (\varepsilon, \Lambda)}, I_{\varepsilon}, -))\\
    & = & \mathbf{L}^{\mathcal{O} (\mathcal{E}_L) ((\hbar))}_{\hbar} \left( 
    \small{\frac{\partial}{\partial \Lambda'}} S (\mathcal{W} (\partial_{P
    (\Lambda, \Lambda')}, I_{\Lambda}, -)) \right)_{\Lambda' = \Lambda}
    \mathbf{E}^{\mathcal{O} (\mathcal{E}_L) ((\hbar))}_{\hbar}
    \tmmathbf{\kappa}_{\hbar} (\mathcal{W} (\partial_{P (\varepsilon,
    \Lambda)}, I_{\varepsilon}, -))\\
    & = & \tmmathbf{\mathcal{K}}_{\hbar} \left( \left(
    \small{\frac{\partial}{\partial \Lambda'}} \mathcal{W} (\partial_{P
    (\Lambda, \Lambda')}, I_{\Lambda}, -) \right)_{\Lambda' = \Lambda} \right)
    \tmmathbf{\kappa}_{\hbar} (\mathcal{W} (\partial_{P (\varepsilon,
    \Lambda)}, I_{\varepsilon}, -)) .
  \end{eqnarray*}
  Define $G_{\Lambda} \assign \left( \frac{\partial}{\partial \Lambda'}
  \partial_{P (\Lambda, \Lambda')} \right)_{\Lambda' = \Lambda}$, it is a
  $2$-to-$0$ operator on $\mathcal{O} (\mathcal{E}_L)$. Similar to Remark
  \ref{rmk34biabiaddd}, we use $\{ -, - \}_{G_{\Lambda}}$ to denote
  $\mathcal{K} (G_{\Lambda})_2 : \mathcal{O} (\mathcal{E}_L)^{\otimes 2}
  \mapsto \mathcal{O} (\mathcal{E}_L)$, and extend it to
  \[ \{ -, - \}_{G_{\Lambda}} : \mathcal{O}_{\mathcal{P}}^{> 0}
     (\mathcal{E}_{L, c}) [[\hbar]] \otimes \mathcal{O} (\mathcal{E}_L)
     [[\hbar]] \mapsto \mathcal{O} (\mathcal{E}_L) [[\hbar]], \]
  which is a derivation with respect to $\mathcal{O} (\mathcal{E}_L)
  [[\hbar]]$. It is direct to check that
  \[ \left( \small{\frac{\partial}{\partial \Lambda'}} \mathcal{W}
     (\partial_{P (\Lambda, \Lambda')}, I_{\Lambda}, -) \right)_{\Lambda' =
     \Lambda} = \hbar G_{\Lambda} + \{ I_{\Lambda}, - \}_{G_{\Lambda}} . \]
  Thus $\tmmathbf{\kappa}_{\hbar} (\mathcal{W} (\partial_{P (\varepsilon,
  \Lambda)}, I_{\varepsilon}, -))$ satisfies a first order linear differential
  equation
  \[ \small{\frac{\partial}{\partial \Lambda}} \tmmathbf{\kappa}_{\hbar}
     (\mathcal{W} (\partial_{P (\varepsilon, \Lambda)}, I_{\varepsilon}, -))
     =\tmmathbf{\mathcal{K}}_{\hbar} (\hbar G_{\Lambda} + \{ I_{\Lambda}, -
     \}_{G_{\Lambda}}) \tmmathbf{\kappa}_{\hbar} (\mathcal{W} (\partial_{P
     (\varepsilon, \Lambda)}, I_{\varepsilon}, -)) . \]
  Since $\tmmathbf{\mathcal{K}}_{\hbar} (\hbar G_{\Lambda} + \{ I_{\Lambda}, -
  \}_{G_{\Lambda}})$ and $\tmmathbf{\kappa}_{\hbar} (\mathcal{W} (\partial_{P
  (\varepsilon, \varepsilon)}, I_{\varepsilon}, -))$ preserve $S (\mathcal{O}
  (\mathcal{E}_L)) [[\hbar]]$, we conclude that for $\forall \Lambda > 0$,
  $\tmmathbf{\kappa}_{\hbar} (\mathcal{W} (\partial_{P (\varepsilon,
  \Lambda)}, I_{\varepsilon}, -))$ preserves $S (\mathcal{O} (\mathcal{E}_L))
  [[\hbar]]$.
\end{proof}

\subsubsection*{Homotopy transfer for derived BV algebras}

In Remark \ref{rmkdrvdbvalgefteg58} we find that the effective observable
complexes in the examples are all derived BV algebras. Now we give an
explanation.

Let $\mathcal{A}, \mathcal{B}$ be graded commutative unital algebras over
field $k$, suppose there is an SDR of $k [[\hbar]]$-cochain complexes
\begin{equation}
  (\mathcal{B} [[\hbar]], D') \begin{array}{c}
    \mathfrak{i}\\
    \rightleftharpoons\\
    \mathfrak{p}
  \end{array} (\mathcal{A} [[\hbar]], D), \mathfrak{K}, \label{lm513innerfml1}
\end{equation}
where $(\mathcal{A} [[\hbar]], D)$ is a derived BV algebra, $\mathfrak{i}
(\tmmathbf{1}_{\mathcal{B}}) =\tmmathbf{1}_{\mathcal{A}}, \mathfrak{p}
(\tmmathbf{1}_{\mathcal{A}}) =\tmmathbf{1}_{\mathcal{B}}$, and for $\forall n
\geqslant 1, b_1, b_2, \ldots, b_n \in \mathcal{B}$,
\begin{equation}
  \mathfrak{p} (\mathfrak{i} (b_1) \mathfrak{i} (b_2) \cdots \mathfrak{i}
  (b_n)) = b_1 b_2 \cdots b_n . \label{lm513innerfml2}
\end{equation}
By symmetric tensor power construction (Lemma \ref{lemsymsdrcnstrctn}) and $k
((\hbar))$-linear extension, (\ref{lm513innerfml1}) leads to
\begin{equation}
  (S (\mathcal{B}) ((\hbar)), \overline{D'}) \begin{array}{c}
    S (\mathfrak{i})\\
    \rightleftharpoons\\
    S (\mathfrak{p})
  \end{array} (S (\mathcal{A}) ((\hbar)), \overline{D} ),
  \mathfrak{K}^{\tmop{sym}} . \label{lm513innerfml3}
\end{equation}
(It is easy to check that $\mathfrak{i}^{\tmop{sym}} = S (\mathfrak{i}),
\mathfrak{p}^{\tmop{sym}} = S (\mathfrak{p}), (D')^{\tmop{der}} =
\overline{D'}$ and $D^{\tmop{der}} = \overline{D}$, if there is any confusion
here.) There is a conjugation on the RHS of (\ref{lm513innerfml3}):
\begin{equation}
  (S (\mathcal{A}) ((\hbar)), \overline{D} ) \begin{array}{c}
    \mathbf{L}^{\mathcal{A}}_{\hbar}\\
    \rightleftharpoons\\
    \mathbf{E}^{\mathcal{A}}_{\hbar}
  \end{array} (S (\mathcal{A}) ((\hbar)), \tmmathbf{\mathcal{K}}_{\hbar} (D) )
  . \label{lm513innerfml50}
\end{equation}
It is direct to check that $(\tmmathbf{\mathcal{K}}_{\hbar} (D) -
\overline{D})$ is a small perturbation to (\ref{lm513innerfml3}). So, we can
write down a perturbed SDR:
\begin{equation}
  (S (\mathcal{B}) ((\hbar)), (\overline{D'})_{\mathbf{E}\mathbf{L}})
  \begin{array}{c}
    S (\mathfrak{i})_{\mathbf{E}\mathbf{L}}\\
    \rightleftharpoons\\
    S (\mathfrak{p})_{\mathbf{E}\mathbf{L}}
  \end{array} (S (\mathcal{A}) ((\hbar)), \tmmathbf{\mathcal{K}}_{\hbar} (D)
  ), (\mathfrak{K}^{\tmop{sym}})_{\mathbf{E}\mathbf{L}} .
  \label{lm513innerfml5}
\end{equation}
$(\mathcal{A} [[\hbar]], D)$ being a derived BV algebra implies the
perturbation $(\tmmathbf{\mathcal{K}}_{\hbar} (D) - \overline{D})$ preserves
$\mathcal{A} [[\hbar]]$. Hence the perturbed maps automatically satisfy
\begin{eqnarray*}
  (\overline{D'})_{\mathbf{E}\mathbf{L}} (S (\mathcal{B}) [[\hbar]]) \subset S
  (\mathcal{B}) [[\hbar]], &  &
  (\mathfrak{K}^{\tmop{sym}})_{\mathbf{E}\mathbf{L}} (S (\mathcal{A})
  [[\hbar]]) \subset S (\mathcal{A}) [[\hbar]],\\
  S (\mathfrak{i})_{\mathbf{E}\mathbf{L}} (S (\mathcal{B}) [[\hbar]]) \subset
  S (\mathcal{A}) [[\hbar]], &  & S (\mathfrak{p})_{\mathbf{E}\mathbf{L}} (S
  (\mathcal{A}) [[\hbar]]) \subset S (\mathcal{B}) [[\hbar]] .
\end{eqnarray*}
\begin{proposition}
  \label{hopethelastlablngpp513}In above settings, we have the following
  facts:
  \begin{enumerateromancap}
    \item If for $\forall n \geqslant 1, a_1, a_2, \ldots, a_n \in
    \mathcal{A}$, $\mathfrak{p}, \mathfrak{K}$ further satisfy
    \begin{equation}
      \sum_{m = 1}^n (\pm)_{\tmop{Kos}} \mathfrak{p} (a_1 \cdots \mathfrak{K}
      (a_m) \cdots a_n) = 0, \label{lm513innerfml6}
    \end{equation}
    then,
    \begin{equation}
      (\overline{D'})_{\mathbf{E}\mathbf{L}} =\tmmathbf{\mathcal{K}}_{\hbar}
      (D'), \quad S (\mathfrak{p})_{\mathbf{E}\mathbf{L}}
      =\tmmathbf{\kappa}_{\hbar} (\mathfrak{p}) . \label{lm513innerfml8}
    \end{equation}
    So now $(\mathcal{B} [[\hbar]], D')$ is a derived BV algebra, and
    $\mathfrak{p}$ is a morphism of derived BV algebras.
    
    \item If for $\forall n \geqslant 1, b_1, b_2, \ldots, b_n \in
    \mathcal{B}$, $\mathfrak{i}, \mathfrak{K}$ further satisfy
    \begin{equation}
      \mathfrak{K} (\mathfrak{i} (b_1) \mathfrak{i} (b_2) \cdots \mathfrak{i}
      (b_n)) = 0, \label{lm513innerfml7}
    \end{equation}
    then,
    \begin{equation}
      (\overline{D'})_{\mathbf{E}\mathbf{L}} =\tmmathbf{\mathcal{K}}_{\hbar}
      (D'), \quad S (\mathfrak{i})_{\mathbf{E}\mathbf{L}}
      =\tmmathbf{\kappa}_{\hbar} (\mathfrak{i}) . \label{lm513innerfml9}
    \end{equation}
    So now $(\mathcal{B} [[\hbar]], D')$ is a derived BV algebra, and
    $\mathfrak{i}$ is a morphism of derived BV algebras.
  \end{enumerateromancap}
\end{proposition}

\begin{proof}
  For the perturbation to (\ref{lm513innerfml3}) induced by conjugation
  (\ref{lm513innerfml50}), it is direct to check that (\ref{lm513innerfml2})
  and (\ref{lm513innerfml6}) imply the statement \ref{conjupropprojB} of
  Proposition \ref{conjuptbtnprojcomcase}. Similarly, (\ref{lm513innerfml2})
  and (\ref{lm513innerfml7}) imply the statement \ref{conjupropinjB} of
  Proposition \ref{conjuptbtninjcomcase}. Then, these two propositions will
  lead to (\ref{lm513innerfml8}) and (\ref{lm513innerfml9}), respectively.
  Details are left as an exercise.
\end{proof}

\begin{definition}
  Let $\mathcal{A}, \mathcal{B}$ be graded commutative unital algebras,
  suppose there is an SDR of cochain complexes
  \begin{equation}
    (\mathcal{B}, D_{\mathcal{B}}) \begin{array}{c}
      \mathfrak{i}\\
      \rightleftharpoons\\
      \mathfrak{p}
    \end{array} (\mathcal{A}, D_{\mathcal{A}}), \mathfrak{K},
    \label{smflalgctrctn519}
  \end{equation}
  where $D_{\mathcal{B}} (\tmmathbf{1}_{\mathcal{B}}) = 0, D_{\mathcal{A}}
  (\tmmathbf{1}_{\mathcal{A}}) = 0$. We call this SDR a {\tmstrong{semifull
  algebra contraction}} if the following identities are satisfied for all
  $a_1, a_2 \in \mathcal{A}, b_1, b_2 \in \mathcal{B}$:
  \begin{eqnarray*}
    \mathfrak{K} (\mathfrak{K} (a_1) \mathfrak{K} (a_2)) =\mathfrak{K}
    (\mathfrak{K} (a_1) \mathfrak{i} (b_2)) & = & \mathfrak{K} (\mathfrak{i}
    (b_1) \mathfrak{i} (b_2)) =\mathfrak{K} (\tmmathbf{1}_{\mathcal{A}}) = 0\\
    \mathfrak{p} (\mathfrak{K} (a_1) \mathfrak{K} (a_2)) =\mathfrak{p}
    (\mathfrak{K} (a_1) \mathfrak{i} (b_2)) & = & 0, \qquad \mathfrak{p}
    (\mathfrak{i} (b_1) \mathfrak{i} (b_2)) = b_1 b_2, \qquad \mathfrak{p}
    (\tmmathbf{1}_{\mathcal{A}}) =\tmmathbf{1}_{\mathcal{B}} .
  \end{eqnarray*}
  Note that these imply $\mathfrak{i} (\tmmathbf{1}_{\mathcal{B}})
  =\mathfrak{i}\mathfrak{p} (\tmmathbf{1}_{\mathcal{A}}) = (1
  +\mathfrak{K}D_{\mathcal{A}} + D_{\mathcal{A}} \mathfrak{K})
  (\tmmathbf{1}_{\mathcal{A}}) =\tmmathbf{1}_{\mathcal{A}}$.
\end{definition}

\begin{proposition}
  Given a semifull algebra contraction as (\ref{smflalgctrctn519}), the
  following identities are satisfied for $\forall n \geqslant 1, b_1, b_2,
  \ldots, b_n \in \mathcal{B}$:
  \[ \mathfrak{p} (\mathfrak{i} (b_1) \mathfrak{i} (b_2) \cdots \mathfrak{i}
     (b_n)) = b_1 b_2 \cdots b_n, \qquad \mathfrak{K} (\mathfrak{i} (b_1)
     \mathfrak{i} (b_2) \cdots \mathfrak{i} (b_n)) = 0. \]
\end{proposition}

\begin{proof}
  Hint: use $\mathfrak{i} (b_1) \cdots \mathfrak{i} (b_n) =
  (\mathfrak{i}\mathfrak{p}-\mathfrak{K}D_{\mathcal{A}} - D_{\mathcal{A}}
  \mathfrak{K}) (\mathfrak{i} (b_1) \cdots \mathfrak{i} (b_n))$ and induction
  on $n$.
\end{proof}

\begin{proposition}
  Given a semifull algebra contraction as (\ref{smflalgctrctn519}), if we
  perturb it by a small perturbation $\delta_{\mathcal{A}}$ satisfying
  $\delta_{\mathcal{A}} (\tmmathbf{1}_{\mathcal{A}}) = 0$, the resulting
  perturbed SDR will also be a semifull algebra contraction.
\end{proposition}

\begin{proof}
  By direct check.
\end{proof}

\begin{remark}
  The result (\ref{tnsrpwrsdr}) of symmetric tensor power construction is a
  semifull algebra contraction.
\end{remark}

So, we have the following conclusion:

\begin{proposition}
  Let $\mathcal{A}, \mathcal{B}$ be graded commutative unital algebras over
  field $k$, suppose there is a semifull algebra contraction of $k
  [[\hbar]]$-algebras
  \begin{equation}
    (\mathcal{B} [[\hbar]], D_{\mathcal{B}}) \begin{array}{c}
      \mathfrak{i}\\
      \rightleftharpoons\\
      \mathfrak{p}
    \end{array} (\mathcal{A} [[\hbar]], D_{\mathcal{A}}), \mathfrak{K},
    \label{endendendfml520}
  \end{equation}
  where $(\mathcal{A} [[\hbar]], D_{\mathcal{A}})$ is a derived BV algebra,
  $\mathfrak{i} (\tmmathbf{1}_{\mathcal{B}}) =\tmmathbf{1}_{\mathcal{A}},
  \mathfrak{p} (\tmmathbf{1}_{\mathcal{A}}) =\tmmathbf{1}_{\mathcal{B}}$.
  Then, $(\mathcal{B} [[\hbar]], D_{\mathcal{B}})$ is also a derived BV
  algebra, and the injection $\mathfrak{i}$ is a morphism of derived BV
  algebras.
  
  If we perturb (\ref{endendendfml520}) by a small perturbation
  $\delta_{\mathcal{A}}$ such that $(\mathcal{A} [[\hbar]], D_{\mathcal{A}} +
  \delta_{\mathcal{A}})$ is still a derived BV algebra, then the perturbed
  differential on $\mathcal{B} [[\hbar]]$ endows $\mathcal{B} [[\hbar]]$ a
  derived BV algebra structure, and the perturbed injection is still a
  morphism of derived BV algebras.
\end{proposition}

Finally, we come to our constructions in previous sections:

\begin{proposition}
  \label{reallythelastpp519ref}The SDR's (\ref{endendend217real}) and
  (\ref{ptbtnrsttqmopnclsd44}) are homotopies between derived BV algebras,
  i.e., the injections and projections are morphisms of derived BV algebras.
\end{proposition}

\begin{proof}
  By the previous proposition we know the injections in
  (\ref{endendend217real}) and (\ref{ptbtnrsttqmopnclsd44}) are morphisms of
  derived BV algebras. As for the projections, (\ref{ppgtvanishkfactorcdt})
  implies the maps in (\ref{endendend217real}) satisfy
  \[ p_{\hbar}^{\tmop{int}} (a_1 K_{\hbar}^{\tmop{int}} (a_2)) = 0 \qquad
     \text{for } \forall a_1, a_2 \in \widehat{\tmop{Sym}} (M) [[\hbar]] . \]
  Similarly (by de Rham degree reason) the maps in
  (\ref{ptbtnrsttqmopnclsd44}) satisfy
  \[ p_t (a_1 K_t^{\tmop{int}} (a_2)) = 0 \qquad \text{for } \forall a_1, a_2
     \in \mathcal{O} (\mathcal{E}_L) [[\hbar]] . \]
  Hence both (\ref{endendend217real}) and (\ref{ptbtnrsttqmopnclsd44}) satisfy
  the condition (\ref{lm513innerfml6}). By Proposition
  \ref{hopethelastlablngpp513}, the projections in (\ref{endendend217real})
  and (\ref{ptbtnrsttqmopnclsd44}) are also morphisms of derived BV algebras.
\end{proof}

\begin{remark}
  As commutative algebras, the observable complexes in this paper are quite
  special: they all have the form of $\widehat{\tmop{Sym}} (V) [[\hbar]]$ for
  some $V$. Inspired by {\cite[Section 3]{bandiera2020cumulants}}, if
  $(\widehat{\tmop{Sym}} (V) [[\hbar]], D)$ is a derived BV algebra, we may
  call it a ``dual $\tmop{IBL}_{\infty} [1]$ algebra''. We might modify
  Theorem 3.6 in {\cite{bandiera2020cumulants}} and use it to prove
  Proposition \ref{reallythelastpp519ref} without refering to the condition
  (\ref{lm513innerfml6}). We leave this consideration for later work.
\end{remark}

\end{document}